\documentclass[12pt]{article}
\usepackage{amsmath,amsthm,amscd,amssymb}
\usepackage[margin=.75 in]{geometry}
\usepackage{graphicx}
\usepackage{hyperref}
\usepackage{color}
\usepackage{caption}
\usepackage{diagrams}

\usepackage[normalem]{ulem}

\newcommand{\N}{\mathbb{N}}
\newcommand{\Z}{\mathbb{Z}}

\newcommand{\R}{\mathbb{R}}
\newcommand{\C}{\mathbb{C}}
\newcommand{\E}{\mathbb{E}}
\newcommand{\ga}{\tau}
\newcommand{\ro}{{\color{black} \otimes}}

\newcommand{\Hyp}{\ _1F_1}

\newtheorem{theorem}{Theorem}[section]
\newtheorem{proposition}[theorem]{Proposition}
\newtheorem{lemma}[theorem]{Lemma}

\newcounter{examplecounter}
\newenvironment{example}{{\noindent\ignorespaces}%
{\par\noindent%
\ignorespacesafterend}
    \refstepcounter{examplecounter}%
  \textbf{Example \arabic{examplecounter}}%
  \quad
}{

}
\newenvironment{remark}{{\noindent\ignorespaces}%
{\par\noindent%
\ignorespacesafterend}
  \textbf{Remark.}%
}{

}

\newcommand{\Uq}{\mathcal{U}_q(\mathfrak{gl}_n)}
\newcommand{\eps}{\epsilon}
\newcommand{\state}[1]{\left\langle {#1} \right\rangle}

\title{A (2+1)--dimensional Gaussian field as asymptotic fluctuations of quantum random walks on quantum groups}
\author{Jeffrey Kuan}
\date{}

\begin{document}
\maketitle

\abstract{This paper introduces a (2+1)--dimensional Gaussian field which has the Gaussian free field on the upper half--plane with zero boundary conditions as certain two--dimensional sections. Along these sections, called space--like paths, it matches the Gaussian field from eigenvalues of random matrices and from a growing random surface. However, along time--like paths the behavior is different.

The Gaussian field arises as the asymptotic fluctuations in quantum random walks on quantum groups $\Uq$. This quantum random walk is a $q$--deformation of previously considered quantum random walks. When restricted to the space--like paths, the moments of the quantum random walk match the moments of the growing random surface.

}

\section{Introduction}

The Gaussian free field (GFF) is a two--dimensional Gaussian field which arises as asymptotics in many probabilistic models. See \cite{S} for a mathematical introduction to the GFF. For time--dependent models it is natural to ask if there is a canonical (2+1)--dimensional Gaussian field generalizing the GFF. For a random surface growth model \cite{BF}, the fluctuations along space--like paths  (that is, paths in space--time where time increases as the vertical co--ordinate decreases) were shown to be the Gaussian free field -- however, the behavior along time--like paths was inaccessible. In a later paper \cite{Bo} which analyzed eigenvalues of corners of time--dependent random matrices, the resulting asymptotics were shown to be a time--dependent (2+1)--dimensional Gaussian field $\mathfrak{G}$ whose restrictions to space--like paths are the GFF, and matches the \cite{BF} asymptotics along space--like paths.  Additional work looking at the edge of this model was also done in \cite{So}. These were partially motivated by physics literature, which predicted the GFF as the stationary distribution for the Anisotropic Kardar--Parisi--Zhang equation (see e.g. \cite{KPZ,KS,W}).  

In \cite{K}, a quantum random walk was constructed whose moments match the random surface growth model along space--like paths, and again the field $\mathfrak{G}$ arises in the asymptotics, after applying the standard Brownian Motion to Ornstein--Uhlenbeck rescaling. It is natural to ask if $\mathfrak{G}$ is the only canonical time--dependent Gaussian field having the GFF as fixed--time marginals (a very rough analogy would be the characterization of the Ornstein--Uhlenbeck process as the only Gaussian, stationary, Markov process). However, this paper will construct a different field $\mathcal{G}$.

The field $\mathcal{G}$ will arise as fluctuations of a quantum random walk on quantum groups (QRWQG), which is a variant of \cite{B,Bi2,BB,CD, K,K1}.  As was the case in \cite{K}, along space--like paths the moments of the QRWQG are precisely the same as the moments for the random surface growth model from \cite{BF}. This shows that $\mathfrak{G}$ and $\mathcal{G}$ are identical along space--like paths, but it turns out that they are not the same along time--like paths. 

Having introduced $\mathcal{G}$, now turn the discussion to the quantum random walks. Recall that the motivation for quantum random walks comes from quantum mechanics. Rather than defining a state space as a \textit{set} of states, instead the state space is a \textit{Hilbert space} of wavefunctions. The observables, rather than being functions on the state space, are operators on this Hilbert space. These operators are related to classical observables through their eigenvalues. Generally, observables do not have to commute, so for this reason quantum random walks are also called non--commutative random walks. The randomness occurs through $states$, which are linear functionals on the space of observables, corresponding to the expectation of the observable.

Before describing the quantum version of the random surface, first review how it looks in the classical viewpoint. On each horizontal section of the random stepped surface, the (classical) state space is the set
$$
 \mathbb{GT}_n:=\{ \lambda=(\lambda_1 \geq \ldots \geq \lambda_n): \lambda_i \in \mathbb{Z}\}
$$
By analogy with quantum mechanics, the space of wavefunctions should consist of functions of the form $\chi_{\vec{x}}(\lambda)$. By dimension considerations, $\vec{x}$ should vary over $\mathbb{C}^n$. Switching indices and variables, write $\chi_{\lambda}(x_1,\dots,x_n)$, so the wave functions are some class of functions on $\mathbb{C}^n.$ The observables are then some space of operators on this class of functions. Any probability measure $\mathbb{P}(\lambda)$ on the $\lambda$ can be encoded through $\chi = \sum_{\lambda} \mathbb{P}(\lambda)\chi_{\lambda}$. Furthermore, if $D$ is an operator for which $\{\chi_{\lambda}: \lambda \in \mathbb{GT}_n\}$ are eigenfunctions with eigenvalues $a(\lambda)$, then
$$
(D\chi)(1,\ldots,1) = \sum_{\lambda} \mathbb{P}(\lambda) a(\lambda)\chi_{\lambda}(1,\ldots,1) = \mathbb{E}_{\mathbb{P}} \left[ a(\lambda) \cdot \chi_{\lambda}(1,\ldots,1)\right]
$$
The $\chi_{\lambda}$ can be chosen so that $\chi_{\lambda}(1,\ldots,1)$ is normalized to $1$. The phenomenon that a state can be defined from a wave function can be seen as an analog of the Gelfand--Naimark--Segal construction.

The actual construction of the observables and $\chi_{\lambda}$ comes from representation theory. In particular, the set $\mathbb{GT}_n$ parameterizes the highest weights of finite--dimensional irreducible representations of $\mathfrak{gl}_n$, and it is through these representations that the observables are defined. Here, the (non--commutative) space of observables  is the Drinfeld--Jimbo quantum group $\Uq$. Each $u\in \Uq$ has a corresponding difference operator from \cite{EK}.  The relevant observables for asymptotics are certain central elements $C_q^{(n)} \in Z(\Uq)$ calculated in \cite{GZB}. The eigenvalue of each $C_q^{(n)}$ on the irreducible representation $V_{\lambda}$ is $\sum_{i=1}^n q^{2(\lambda_i-i+n)}$, so one can think of these observables as linear statistics of the function $q^{2x}$.  

There are a few key differences between this construction and previous constructions that are worth highlighting.  When $q\rightarrow 1$, the QRWQG reduces to the quantum random walk in \cite{BB,K}, which used the universal enveloping algebra $\mathcal{U}(\mathfrak{gl}_n)$ as the space of observables. The papers there used differential operators on the Lie group $GL(n)$ to define the quantum random walk. The relevant central elements in $\mathcal{U}(\mathfrak{gl}_n)$  acted as $\sum_{i=1}^n (\lambda_i-i+n)^k$ for $k\geq 1$, so can be thought of as linear statistics of $x^k$. In that case, the fluctuations of linear statistics for different values of $k$ were computed, which suggests finding the fluctuations of these linear statistics for different $q$ here.

Additionally, $\Uq$ is no longer co--commutative as a Hopf algebra. (The word ``quantum'' appearing twice in the title of this paper refers to two different meanings: the first one makes the space of observables non--commutative, and the second makes it non-co--commutative). Probabilistically, this results in different dynamics, which ultimately leads to the Gaussian fields differing along time--like paths. However, the functions $\chi_{\lambda}$ do not depend on $q$, which is why the Gaussian fields match along space--like paths.

We also mention several algebraic reasons for taking the approach with quantum groups. One is that in the $q\rightarrow 1$ limit, the asymptotics are dependent on Schur-Weyl duality (see equation (2.9) of \cite{BB}, which references Proposition 3.7 of \cite{IO}), so does not generalize to other Lie algebras. Additionally, the relevant central elements are actually easier to construct in the quantum case than in the classical case (for example, see section 7.5 of  \cite{GKLLRT} or chapter 7 of \cite{kn:M} for explicit central elements of $\mathcal{U}(\mathfrak{gl}_n)$). Another notable difference is that the non--commutative Markov operator $P_t$ no longer preserves the center when $q\neq 1$. However, (somewhat surprisingly) each $P_t C_q^{(n)}$ can be written as a linear combination of $C_q^{(n-k)}$, generalizing a result of \cite{Bi3} for $n=2$. So this paper demonstrates (in a sense) that preserving the center is not necessary for developing meaningful asymptotics. However, note that one would not necessarily expect $P_t C_q^{(n)}$ to be a linear combination of $C_q^{(n-k)}$ in every quantum group, so while it should be possible to construct a QRWQG in general, the asymptotics may be more difficult.

Finally, it is important to mention that some of the cited papers actually prove more than what is needed here. For example, \cite{Bo} and \cite{BB} actually show convergence to correlated Gaussian free fields. The difference operators in \cite{EK} can be used to construct Macdonald's difference operators for all $(q,t)$, and here only the $q=t$ case is used. Therefore, it should be possible to extend the results of this paper to more generality.

Let us outline the body of the paper. In section \ref{3DGFF}, we define the Gaussian field $\mathcal{G}$ and show by direct computation its relationship to $\mathfrak{G}$ along space--like paths. Section \ref{Back} reviews some of the necessary definitions in non--commutative probability theory and representation theory. Section \ref{QRWonQG} provides the construction of the quantum random walk. In section \ref{Ctcp}, the random surface growth from \cite{BF} is defined and shown to have the same expectations as the quantum random walk along space--like paths. Finally, section \ref{Asympt} shows that the fluctuations in the QRWQG converge to $\mathcal{G}$.

\textbf{Acknowledgements.}  The author is grateful for enlightening conversations with Alexei Borodin, Ivan Corwin, Philippe Biane, Pavel Etingof and Yi Sun. The author would also like to thank Mark Cerenzia for some of the calculations. Financial support was available through NSF grant DMS-1502665. Additional financial support was available through NSF grant DMS-1302713 and the Fields Institute, which allowed the author to attend the workshop ``Focus Program on Noncommutative Distributions in Free Probability Theory.''

\section{A (2+1)--dimensional Gaussian field}\label{3DGFF}
Let $\mathfrak{G}$ be a Gaussian field indexed by $(k,\eta,\tau) \in \mathbb{N} \times \mathbb{R}_{>0} \times \mathbb{R}_{>0}$  with mean zero and covariance given by
\begin{multline*}
\mathbb{E}[\mathfrak{G}(k_i,\eta_i,\tau_i)\mathfrak{G}(k_j,\eta_j,\tau_j)]\\
= 
\begin{cases}
\displaystyle\left(\frac{1}{2\pi i}\right)^2\iint_{\vert z\vert>\vert w\vert}\limits (\eta_i z^{-1} + \tau_i + \tau_i z)^{k_i} (\eta_j w^{-1} + \tau_j + \tau_j w)^{k_j} (z-w)^{-2}dzdw,&\\
\hspace{3in} \eta_i\geq\eta_j,\ga_i\leq\ga_j&\\
\displaystyle\left(\frac{1}{2\pi i}\right)^2\iint_{\vert z\vert>\vert w\vert}\limits (\eta_j\frac{\tau_j}{\tau_i} z^{-1} + \tau_j + \tau_i z)^{k_j} (\eta_i w^{-1} + \tau_i + \tau_i w)^{k_i} (z-w)^{-2}dzdw,&\\
\hspace{3in} \eta_i<\eta_j,\ga_i\leq\ga_j&
\end{cases}
\end{multline*}
where the $z,w$ contours are counterclockwise circles centered around the origin. As explained in \cite{K}, $\mathfrak{G}$ can be viewed as the moments of a three--dimensional Gaussian field which has the Gaussian free field as fixed--time marginals. 

Let $\mathcal{G}$ be a Gaussian field indexed by $(h,\eta,\tau) \in \mathbb{R} \times \mathbb{R}_{>0} \times \mathbb{R}_{>0}$ with mean zero and covariance $\mathbb{E}[\mathcal{G}(\tilde{h},\eta_i,\tau_i)\mathcal{G}(h,\eta_j,\tau_j)] $ 
given by: 
\begin{multline}
\text{if } \tau_j\geq \tau_i,\eta_j\leq \eta_i ,\quad \left( \frac{1}{2\pi i}\right)^2 e^{2\tilde{h}\tau_i}e^{2h\tau_j} \displaystyle \iint_{\vert z\vert > \vert w\vert} \exp \left( 2\tilde{h} (\eta_i z^{-1} + \tau_i z) \right) (z-w)^{-2}dzdw \\
\times \Bigg(\exp  ( 2h (\eta_j w^{-1} + \tau_i w)) + \frac{2h \sqrt{\tau_j-\tau_i}}{2\pi i} \int_0^{\eta_j} e^{2h((\eta_j-\kappa)w^{-1} + \tau_i w)} d\kappa \oint e^{2h\sqrt{\tau_j-\tau_i}\sqrt{\kappa}(t +t^{-1}))}\frac{t^{-2}}{\sqrt{\kappa}}dt \Bigg),
\end{multline}
\begin{multline}
\text{if } \tau_j\geq \tau_i,\eta_j \geq \eta_i ,\quad \left( \frac{1}{2\pi i}\right)^2 e^{2\tilde{h}\tau_i}e^{2h\tau_j} \displaystyle \iint_{\vert z\vert > \vert w\vert} \exp \left( 2\tilde{h} (\eta_i z^{-1} + \tau_i z) \right) (z-w)^{-2}dzdw \\
\times \Bigg(\exp  ( 2h (\eta_j w^{-1} + \tau_i w)) + \frac{2h \sqrt{\tau_j-\tau_i}}{2\pi i} \int_{\eta_j-\eta_i}^{\eta_j} e^{2h((\eta_j-\kappa)w^{-1} + \tau_i w)} d\kappa \oint e^{2h\sqrt{\tau_j-\tau_i}\sqrt{\kappa}(t +t^{-1}))}\frac{t^{-2}}{\sqrt{\kappa}}dt \Bigg)\\
+ \left( \frac{1}{2\pi i}\right)^2 e^{2\tilde{h}\tau_i}e^{2h\tau_j} \displaystyle \iint_{\vert z\vert < \vert w\vert} \exp \left( 2\tilde{h} (\eta_i z^{-1} + \tau_i z) \right) (z-w)^{-2}dzdw \\
\times \Bigg(  \frac{2h \sqrt{\tau_j-\tau_i}}{2\pi i} \int_0^{\eta_j-\eta_i} e^{2h((\eta_j-\kappa)w^{-1} + \tau_i w)} d\kappa \oint e^{2h\sqrt{\tau_j-\tau_i}\sqrt{\kappa}(t +t^{-1}))}\frac{t^{-2}}{\sqrt{\kappa}}dt \Bigg),
\end{multline}
where the $t,z,w$ contours are counterclockwise circles centered around the origin, with $ \vert t \vert > \vert w\vert\sqrt{\tau_j-\tau_i}$. Note that if $\tau_i\leq \tau_j$ and $\eta_j \leq \eta_i$ then
\begin{multline}\label{SameToDifferent}
\mathbb{E}[\mathcal{G}(\tilde{h},\eta_i,\tau_i)\mathcal{G}(h,\eta_j,\tau_j)] = e^{2h(\tau_j-\tau_i)}\mathbb{E}[\mathcal{G}(\tilde{h},\eta_i,\tau_i)\mathcal{G}(h,\eta_j,\tau_i)] \\
+ e^{2h(\tau_j-\tau_i)} \frac{2h\sqrt{\tau_j-\tau_i}}{2\pi i} \int_0^{\eta_j}  \mathbb{E}[\mathcal{G}(\tilde{h},\eta_i,\tau_i)\mathcal{G}(h,\eta_j-\kappa,\tau_i)] d\kappa \oint e^{2h\sqrt{\tau_j-\tau_i}\sqrt{\kappa}(t+t^{-1})} \frac{t^{-2}}{\sqrt{\kappa}}dt 
\end{multline}

A priori, it is not obvious that $\mathcal{G}$ is a well--defined family of random variables: for instance, the covariance matrix might not be positive--definite. However, it will be shown in Theorem \ref{Asymp} below that $\mathcal{G}$ occurs as the limit of well--defined random variables. Furthermore, numerical computations indicate that the covariance matrices are positive--definite anyway.

The next proposition shows that $\mathcal{G}$ and $\mathfrak{G}$ match along space--like paths.  It also follows from later results (namely, that $\mathcal{G}$ is the limit of the QRWQG, the QRWQG matches the surface growth along space--like paths, and $\mathfrak{G}$ is the limit of the surface growth), but a more elementary proof is provided here. Because $\mathcal{G}$ will appear in the linear statistics of $q^{2x}$ and $\mathfrak{G}$ appears as linear statistics of $x^k$, setting $q=e^h$ motivates the comparison. 

\begin{proposition}\label{Same}
If $\eta_i\geq\eta_j$ and $\tau_i\leq \tau_j$ then
\begin{equation}\label{Matching}
\mathbb{E}[\mathcal{G}(\tilde{h},\eta_i,\tau_i)\mathcal{G}(h,\eta_j,\tau_j)]  = \sum_{k_i,k_j=0}^{\infty} \frac{(2\tilde{h})^{k_i} (2h)^{k_j}}{k_i! k_j!} \mathbb{E}[\mathfrak{G}(k_i,\eta_i,\tau_i)\mathfrak{G}(k_j,\eta_j,\tau_j)]
\end{equation}

If $\tau_j\geq \tau_i, \eta_j\geq \eta_i$, then in general \eqref{Matching} does not hold.
\end{proposition}
\begin{proof}

Assume that $\eta_i\geq\eta_j$ and $\tau_i\leq \tau_j$. By making the substitution $t \mapsto t/\sqrt{\kappa}$, the expression becomes
\begin{multline}
\left( \frac{1}{2\pi i}\right)^2 e^{2\tilde{h}\tau_i}e^{2h\tau_j} \displaystyle \iint_{\vert z\vert > \vert w\vert} \exp \left( 2\tilde{h} (\eta_i z^{-1} + \tau_i z) \right) (z-w)^{-2}dzdw \\
\times \Bigg(\exp  ( 2h (\eta_j w^{-1} + \tau_i w)) + \frac{2h \sqrt{\tau_j-\tau_i}}{2\pi i} \int_0^{\eta_j} e^{2h((\eta_j-\kappa)w^{-1} + \tau_i w)} d\kappa \oint e^{2h\sqrt{\tau_j-\tau_i}(t + \kappa t^{-1}))}t^{-2}dt \Bigg) 
\end{multline}

The integrand in $\kappa$ is simply an exponential function, so evaluates to 
\begin{multline}
\left( \frac{1}{2\pi i}\right)^2 e^{2\tilde{h}\tau_i}e^{2h\tau_j} \displaystyle \iint_{\vert z\vert > \vert w\vert} \exp \left( 2\tilde{h} (\eta_i z^{-1} + \tau_i z) \right) (z-w)^{-2}dzdw \\
\times \Bigg(\exp  ( 2h (\eta_j w^{-1} + \tau_i w)) + \frac{2h \sqrt{\tau_j-\tau_i}}{2\pi i}  \oint\Big( \frac{e^{2h \eta_j (\sqrt{\tau_j-\tau_i}t^{-1}-w^{-1})}-1}{2h(\sqrt{\tau_j-\tau_i}t^{-1}-w^{-1}) } \Big)   e^{2h(\eta_jw^{-1} + \tau_i w)} e^{2ht\sqrt{\tau_j-\tau_i}}t^{-2}dt \Bigg) 
\end{multline}
Substitute $t\mapsto \sqrt{\tau_j-\tau_i}t$
\begin{multline}
\left( \frac{1}{2\pi i}\right)^2e^{2\tilde{h}\tau_i}e^{2h\tau_j}  \displaystyle \iint_{\vert z\vert > \vert w\vert} \exp \left( 2\tilde{h} (\eta_i z^{-1} + \tau_i z) \right) (z-w)^{-2}dzdw \\
\times \Bigg(\exp  ( 2h (\eta_j w^{-1} + \tau_i w)) + \frac{2h}{2\pi i}  \oint\Big( \frac{e^{2h \eta_j (t^{-1}-w^{-1})}-1}{2ht(1-tw^{-1}) } \Big)   e^{2h(\eta_jw^{-1} + \tau_i w)} e^{2ht(\tau_j-\tau_i)}dt \Bigg) 
\end{multline}
By the assumptions on the $t$ and $w$ contours,
$$
\frac{2h}{2\pi i}  \oint\Big( \frac{-1}{2ht(1-tw^{-1}) } \Big)   e^{2ht(\tau_j-\tau_i)}dt  = -1 +   e^{2h w(\tau_j-\tau_i)}
$$
leaving us with
\begin{multline}
\left( \frac{1}{2\pi i}\right)^2 e^{2\tilde{h}\tau_i}e^{2h\tau_j}  \displaystyle \iint_{\vert z\vert > \vert w\vert} \exp \left( 2\tilde{h} (\eta_i z^{-1}+ \tau_i z) \right) (z-w)^{-2}dzdw \\
\times \Bigg(\exp  ( 2h (\eta_j w^{-1} + \tau_j w)) + \frac{1}{2\pi i}  \oint\Big( \frac{e^{2h \eta_j (t^{-1}-w^{-1})}}{t(1-tw^{-1}) } \Big)   e^{2h(\eta_jw^{-1} + \tau_i w)} e^{2ht(\tau_j-\tau_i)}dt \Bigg) 
\end{multline}

So it remains to check that 
\begin{multline}
\left( \frac{1}{2\pi i}\right)^2 e^{2\tilde{h}\tau_i}e^{2h\tau_j}   \displaystyle \iint_{\vert z\vert > \vert w\vert} \exp \left( 2\tilde{h} (\eta_i z^{-1} +\tau_i z) \right) (z-w)^{-2}dzdw \\
\times \Bigg(\frac{1}{2\pi i}  \oint\Big( \frac{e^{2h \eta_j (t^{-1}-w^{-1})}}{t(1-tw^{-1}) } \Big)   e^{2h(\eta_jw^{-1} + \tau_i w)} e^{2ht(\tau_j-\tau_i)}dt \Bigg) =0
\end{multline}
But this follows immediately, because the $w^{-1}$ terms in the exponential cancel, so the integrand has no residues in $w$. So \eqref{Matching} is true.

Now suppose that $\tau_j\geq \tau_i, \eta_j\geq \eta_i$. The $\tilde{h} h^3$ coefficient of the right--hand--side of \eqref{Matching} is
$$
\frac{2\tilde{h}(2h)^3 }{6} \cdot \eta_i \cdot \left( 3\tau_i\tau_j^2 + 3\eta_j\tau_j\tau_i \right) = 8\tilde{h}h^3 \eta_i\tau_i\tau_j(\tau_j + \eta_j)
$$
But on the left--hand--side it is
\begin{multline*}
\frac{16}{3!}\eta_j \cdot \left( 3\tau_i\cdot \tau_i^2 + 3\eta_i\tau_i\cdot \tau_i \right) + 2\sqrt{\tau_j-\tau_i} \int_0^{\eta_j} 4\tau_i\text{min}(\eta_j-\kappa,\eta_i) d\kappa \cdot 2\sqrt{\tau_j-\tau_i}\\
= \frac{16}{6}\eta_j \cdot \left( 3\tau_i\cdot \tau_i^2 + 3\eta_i\tau_i\cdot \tau_i \right) + 16(\tau_j-\tau_i)\tau_i\left( -\tfrac{1}{2}(\eta_j-\kappa)^2 \vert_{\kappa=\eta_j-\eta_i}^{\eta_j}+\eta_i(\eta_j-\eta_i)\right)\\
=8 \eta_j \cdot \left( \tau_i\cdot \tau_i^2 + \eta_i\tau_i\cdot \tau_i \right) + 16(\tau_j-\tau_i)\tau_i\left( \tfrac{1}{2}\eta_i^2 +\eta_i(\eta_j-\eta_i)\right)  
\end{multline*}
which is not equal to the expression above.
\end{proof}

\section{Background Definitions}\label{Back}

\subsection{Non--commutative probability}\label{Ncp}
Here are some of the basic definitions of objects in non-commutative probability. A more comprehensive introduction can be found in \cite{B2}.

A non--commutative probability space $(W,\omega)$ is a unital $^*$--algebra 
$W$ with identity $1$ and a state $\omega:W\rightarrow\C$, that is, a linear map such that $\omega(a^*a)\geq 0$ and $\omega(1)=1$. Elements of $W$ are 
called \textit{non--commutative random variables}. 
This generalizes a classical probability space, by considering $W=L^{\infty}(\Omega,\mathcal{F},\mathbb{P})$ with $\omega(X)=\E_{\mathbb{P}}X$. We also need a notion of convergence. For a 
large parameter $L$ and $a_1,\ldots,a_r\in (W,\omega)$ which depend on $L$, as well as a limiting space $(\mathbb{A},\Phi)$, we say that $(a_1,\ldots,a_r)$ converges to $\mathbf{(a_1,\ldots,a_r)}$ with respect to the state $\omega$ if
$$
\omega(a_{i_1}^{\epsilon_1}\cdots a_{i_k}^{\epsilon_k})\rightarrow \Phi(\mathbf{a_{i_1}^{\epsilon_1}\cdots a_{i_k}^{\epsilon_k}})
$$
for any $i_1,\ldots,i_k\in\{1,\ldots,r\},\epsilon_j\in\{1,*\}$ and $k\geq 1$.

There is also a non--commutative version of a Markov chain. If $X_n:(\Omega,\mathcal{F},\mathbb{P})\rightarrow E$ denotes the Markov process with transition operator $P :L^{2}(E)\rightarrow L^{2}(E)$, then the Markov property is
$$
\E[Yf(X_{n+1})]=\E[Y Pf(X_n)]
$$
for $f\in L^{2}(E)$ and $Y$ a $\sigma(X_1,\ldots,X_n)$--measurable random variable. Letting $j_n:L^{2}(E)\rightarrow L^{2}(\Omega,\mathcal{F},\mathbb{P})$ be defined by $j_n(f)=f(X_n)$, we can write the Markov property as
$$
\E[Y j_{n+1}(f)]=\E[Y j_n(Pf)]
$$
for all $f\in L^{2}(E)$ and $Y$ in the subalgebra of $L^{\infty}(\Omega,\mathcal{F},\mathbb{P})$ generated by the images of $j_0,\dots,j_n$. 

Translating into the non--commutative setting, we define a \textit{non--commutative Markov operator} to be a semigroup of unital linear maps $\{P_t:t\in T\}$ from a $^*$--algebra $U$ to itself (not necessarily an algebra morphism). In general, the set $T$ indexing time can be either 
$\N$ or $\R_{\geq 0}$. We also require that for any times $t_0<t_1<\ldots\in T$ there exists algebra morphisms  $j_{n}$ from  $U$ to a non--commutative probability space $(W,\omega)$ such that
$$
\omega(j_{n}(f)w)=\omega(j_{{n-1}}(P_{t_n-t_{n-1}}f)w)
$$
for all $f\in U$ and $w$ in the subalgebra of $W$ generated by the images of $\{j_t:t\leq t_{n-1}\}$.

For this paper, the indexing set will be $\R_{\geq 0}$, but we will fix an increasing sequence of times $t_0 < t_1 < \ldots $. In the commutative framework, the analog is a Markov chain $(X_n)_{n \in \mathbb{N}}$ such that $\mathbb{P}(X_{n+1}=y \vert X_n = x) = P_{t_{n+1}-t_n}(x,y)$, where $(P_t)_{t \in \mathbb{R}_{\geq 0}}$ is a probability semigroup.

\subsection{Representation theory}

\subsubsection{Definition of Quantum Groups}\label{DoQG}

This sub--subsection defines the quantum groups. See \cite{J} for a more thorough treatment. 

Let $q$ be a formal variable. The quantum group $\Uq$ is the Hopf algebra with generators  $\{ E_{i,i+1},E_{i+1,i}: 1\leq i \leq n-1\},\{q^{E_{ii}}:1\leq i\leq n\}$ satisfying the relations (below, $q^{E_{ii}}$ are all invertible and the multiplication is written additively in the exponential, so for example $q^{-E_{11} + 2E_{22}}= \left( q^{E_{11}}\right)^{-1} \left( q^{E_{22}}\right)^2$) 
$$
q^{E_{ii}}q^{E_{jj}} = q^{E_{jj}} q^{E_{ii}} = q^{E_{ii}+E_{jj}}
$$
$$
[E_{i,i+1},E_{i+1,i}] = \frac{q^{E_{ii}-E_{i+1,i+1}}-q^{E_{i+1,i+1}-E_{ii}}}{q-q^{-1}} \quad  \quad  [E_{i,i+1}E_{j+1,j}] =0, \quad i\neq j
$$
\begin{align*}
q^{E_{ii}}E_{i,i+1} &= q E_{i,i+1}q^{E_{ii}} \quad \quad q^{E_{ii}}E_{i-1,i} = q^{-1} E_{i-1,i} q^{E_{ii}}  \quad \quad [q^{E_{ii}},E_{j,j+1}]=0, \quad j\neq i,i-1\\
q^{E_{ii}}E_{i,i-1} &= q E_{i,i-1}q^{E_{ii}} \quad \quad q^{E_{ii}}E_{i+1,i} = q^{-1} E_{i+1,i} q^{E_{ii}} \quad \quad [q^{E_{ii}},E_{j,j-1}]=0, \quad j\neq i,i+1\
\end{align*}
\begin{align*}
E_{i,i+1}^2E_{j,j+1} - (q+q^{-1})E_{i,i+1}E_{j,j+1}E_{i,i+1} +  E_{j,j+1}E_{i,i+1}^2 = 0 , \quad & i = j\pm 1\\
E_{i,i-1}^2E_{j,j-1} - (q+q^{-1})E_{i,i-1}E_{j,j-1}E_{i,i-1} +  E_{j,j-1}E_{i,i-1}^2 = 0 , \quad & i = j\pm 1\\
[E_{i,i+1},E_{j,j+1} ]= 0 = [E_{i,i-1},E_{j,j-1}] , \quad & i \neq j\pm 1
\end{align*}
The co--product is an algebra morphism $\Delta: \Uq \rightarrow \Uq \otimes \Uq$ defined by
\begin{align*}
\Delta\left(q^{E_{ii}}\right) &= q^{E_{ii}} \otimes q^{E_{ii}} \\
\Delta\left( E_{i,i+1} \right) &= q^{E_{ii}-E_{i+1,i+1}} \otimes E_{i,i+1} + E_{i,i+1} \otimes 1 \\
\Delta\left( E_{i,i-1} \right) &= 1 \otimes E_{i,i-1} + E_{i,i-1} \otimes q^{E_{i+1,i+1}-E_{ii}}
\end{align*}
{\color{black}
The multiplication on $\Uq$ is defined by the usual map $m: \Uq^{\otimes 2} \rightarrow \Uq$ sending $v_1  \otimes v_2$
 to $v_1v_2$. By associativity, for every $r\geq 1$ there is a unique algebra morphism $m_r:=m \circ (m_{r-1} \otimes \mathrm{id})=m \circ (\mathrm{id} \otimes m_{r-1})$ from $ \Uq^{\otimes r} $ to $ \Uq$ sending $v_1  \otimes v_2 \otimes \cdots \otimes v_r$ to $v_1 \cdots v_r$.
}

Note that unless $q\rightarrow 1$, $\Delta$ does not satisfy co--commutativity. In other words, if $P$ is the permutation $P(a\otimes b)=b\otimes a$, then $P\circ\Delta \neq \Delta$. However, the co--product does satisfy the co--associativity property
$$
(\mathrm{id} \otimes \Delta) \circ \Delta = (\Delta \otimes \mathrm{id}) \circ \Delta,
$$
so that there is a well--defined algebra morphism $\Delta^{(m-1)} : \Uq \rightarrow \Uq^{\otimes m}$ satisfying $\Delta^{(m)} = (\mathrm{id} \otimes \Delta^{(m-1)}) \circ \Delta = (\Delta^{(m-1)} \otimes \mathrm{id}) \circ \Delta$.  More explicitly,
\begin{align*}
\Delta^{(k-1)}\left(q^{E_{ii}}\right) &= \underbrace{q^{E_{ii}} \otimes q^{E_{ii}} \otimes \cdots \otimes q^{E_{ii}}}_{k}\\
\Delta\left( E_{i,i+1} \right) &= \sum_{i=1}^k \underbrace{ q^{E_{ii}-E_{i+1,i+1}} \otimes \cdots \otimes q^{E_{ii}-E_{i+1,i+1}}}_{i-1} \otimes E_{i,i+1} \otimes \underbrace{1 \otimes \cdots \otimes 1}_{k-i} \\
\Delta\left( E_{i,i-1} \right) &= \sum_{i=1}^k \underbrace{1 \otimes \cdots \otimes 1}_{i-1} \otimes E_{i,i-1} \otimes  \underbrace{ q^{E_{i+1,i+1}-E_{ii}}\otimes q^{E_{i+1,i+1}-E_{ii}}}_{k-i}
\end{align*}
We use the notation $\Delta(u) = \sum_{(u)} u_{1}\otimes u_{2}$. 
This notation will extend to 
$$
\Delta^{(n-1)}(u) = \sum_{(u)}  u_{1} \otimes \cdots \otimes u_{n} 
$$
and
$$
(\mathrm{id} \otimes \Delta) \circ \Delta(u) = \sum_{(u)}\sum_{(u_2)}  u_{1} \otimes u_{21}\otimes u_{22} \quad  \quad (\Delta \otimes \mathrm{id}) \circ \Delta(u) = \sum_{(u)}\sum_{(u_1)}u_{11} \otimes u_{12} \otimes u_{2}
$$
For completeness, the antipode $S$ is an anti--automorphism on $\Uq$ defined on generators by
$$
S(E_{i,i+1}) = -q^{-1} E_{i,i+1} \quad S(E_{i,i-1}) = -q E_{i,i-1}, \quad S(q^{E_{ii}})=q^{-E_{ii}}.
$$
and the co--unit is an algebra morphism $\epsilon: \Uq \rightarrow \mathbb{C}$ defined on generators by
$$
\epsilon\left(q^{E_{ii}}\right)=1 \quad \epsilon\left( E_{i,i\pm 1}\right)=0.
$$
The antipode will only be used in the remark at the end of section \ref{QRWonQG}, and the co--unit will not be used explicitly.

{\color{black}
The quantum group $\Uq$ also has the structure of a $\mathbb{Z}^n$--graded algebra. The grading is defined by setting
\begin{align*}
\deg(q^{E_{ii}})=0, & \quad 1 \leq i \leq n, \\
 \deg(E_{ij})= \epsilon_i - \epsilon_j, & \quad i\neq j,
\end{align*}
where $\epsilon_i = (0,\ldots,0,1,0,\ldots,0)\in \mathbb{Z}^n$, with the $1$ located in the $i$th location. For any $\mu \in \mathbb{Z}^n$, let $\Uq[\mu]$ denote the homogeneous elements with degree $\mu$.
}

For any $1 \leq i \neq j \leq n$, define $E_{ij}$ inductively by
$$
E_{ij} = E_{ik}E_{kj} - q^{-1}E_{kj}E_{ik}, \quad i\lessgtr k \lessgtr j
$$
From the relations defining $\Uq$, it is not hard to see that for $1\leq i<j\leq n$,
\begin{equation}\label{CoProd}
\begin{aligned}
\Delta E_{ij} &= E_{ij} \otimes 1 +q^{E_{ii}-E_{jj}} \otimes E_{ij}+(q - q^{-1})\left [ \sum_{r=i+1}^{j-1} (q^{E_{rr}-E_{jj}} E_{ir}) \otimes E_{rj}  \right ] \\
\Delta E_{ji} &=1 \otimes E_{ji}  + E_{ji} \otimes q^{E_{jj}-E_{ii}} + (q - q^{-1}) \left [ \sum_{r=i+1}^{j-1} E_{ri} \otimes (q^{E_{rr}-E_{ii}}E_{jr})\right ].
\end{aligned}
\end{equation}

The quantum group $\mathcal{U}_q(\mathfrak{gl}_n)$ also carries the structure of a Hopf $^*$--algebra; see \cite{KSBook} {\color{black}or \cite{KoS}} for a more general treatment. Given an involution $^*$ on a Hopf algebra $U$ over $\mathbb{C}$ which makes it a $^*$--algebra, we say that $U$ is a Hopf $^*$--algebra if $\Delta(a^*)=\Delta(a)^*$ and $\epsilon(a^*)=\overline{\epsilon(a)}$ for all $a \in U$. For {\color{black} nonzero} real values of $q$, an explicit expression for the involution $^*$ is given in Proposition 17 of section 6.1.7 of \cite{KSBook}:
$$
\left( q^{E_{ii}}\right)^* = q^{E_{ii}}, \quad E_{i,i+1}^* = q^{E_{ii}-E_{i+1,i+1}}E_{i+1,i}, \quad E_{i+1,i}^* = E_{i,i+1}q^{E_{i+1,i+1}-E_{ii}}.
$$
For any central element $C\in Z(\Uq)$, we have that $C^*=C$.

Note that since the generators of $\Uq$ are a subset the generators of $\mathcal{U}_q(\mathfrak{gl}_m)$ for $n < m$, there is a canonical embedding of $\Uq$ into $\mathcal{U}_q(\mathfrak{gl}_m)$.

\subsubsection{Representations}\label{3.2.2}
The finite--dimensional irreducible representations of $\Uq$ are parameterized by non--increasing sequences of $n$ integers
$$
\mathbb{GT}_n := \{ (\lambda_1 \geq \ldots \geq \lambda_n): \lambda_i \in \mathbb{Z}\}
$$
For each $\lambda \in \mathbb{GT}_n$, let {\color{black} $\pi^{(q)}_{\lambda}: \Uq \rightarrow \mathrm{End}(V_{\lambda})$ denote the corresponding representation. There is a \textit{weight space decomposition}
$$
V_{\lambda} =  \bigoplus_{\mu} V_{\lambda}[\mu]
$$
where $\mu$ is some sequence of integers $\mu=(\mu_1,\ldots,\mu_n)\in \mathbb{Z}^n$ (not necessarily non--increasing) and
$$
V_{\lambda}[\mu] = \{ v \in V_{\lambda}: q^{a_1E_{11} + \ldots + a_nE_{nn}}v = q^{a_1\mu_1 + \ldots + a_n\mu_n} v\}
$$
One can think of the weight spaces as a generalization of eigenspaces. Given any complex numbers $x_1,\ldots,x_n$ there is an action on $V_{\lambda}$ by 
$$
x_1^{E_{11}}\cdots x_n^{E_{nn}} v = x_1^{\mu_1} \cdots x_n^{\mu_n}v \quad \text{ for }  v \in V_{\lambda}[\mu].
$$ 
and write $x^Ev$ for the left--hand--side. With this notation, define the \textit{character} $\chi_{\lambda}$ as
\begin{equation}\label{Chi}
\chi_{\lambda}(x_1,\ldots,x_n) = \mathrm{Tr}\vert_{V_{\lambda}}\left( x^E\right).
\end{equation}
Let $\dim\lambda = \chi_{\lambda}(1,\ldots,1)$ denote the dimension of $V_{\lambda}$. Each $\chi_{\lambda}$ is a symmetric polynomial and the $\{\chi_{\lambda}\}$ form a basis for the ring of symmetric polynomials in $n$ variables. In fact, these are the Schur polynomials $s_{\lambda}$.

The co--product defines the action on tensor products of representations, in the sense that if $v,w$ are vectors in two different representations, then
$$
u \cdot (v\otimes w) = u_1v \otimes u_2w.
$$
In particular,
\begin{equation}\label{Tense}
\pi_{\lambda \otimes \mu}^{(q)}(u) = \pi_{\lambda}^{(q)}(u_1) \otimes \pi_{\mu}^{(q)}(u_2).
\end{equation}

There are also branching rules between representations of $\Uq$ and $\mathcal{U}_q(\mathfrak{gl}_{n-1})$. For $\lambda \in \mathbb{GT}_n$ and $\mu \in \mathbb{GT}_{n-1}$, let $\mu \prec \lambda$ mean 
$$
\mu \prec \lambda \text{   if and only if   } \lambda_1 \geq \mu_1 \geq \lambda_2 \geq \mu_2 \geq \ldots \geq \mu_{n-1} \geq \lambda_n.
$$
If $V_{\lambda}$ is restricted to $\mathbb{GT}_{n-1}$ then it decomposes as 
$$
V_{\lambda} = \bigoplus_{\mu \prec\lambda} V_{\mu}.
$$
More generally, if $m\leq n$, $\lambda^{(m)} \in \mathbb{GT}_m$ and $\lambda^{(n)} \in \mathbb{GT}_n$, then let $m\left(\lambda^{(n)},\lambda^{(m)}\right)$ denote the multiplicities of $V_{\lambda^{(n)}}$ restricted to $\mathbb{GT}_m$, that is
$$
V_{\lambda^{(n)}} = \bigoplus_{\lambda^{(m)} \in \mathbb{GT}_m} m\left(\lambda^{(n)},\lambda^{(m)}\right) V_{\lambda^{(m)}} 
$$
which also means that
$$
\chi_{\lambda^{(n)}}(x_1,\ldots,x_m,1,\ldots,1) = \sum_{\lambda^{(m)} \in \mathbb{GT}_m} m\left(\lambda^{(n)},\lambda^{(m)}\right) \chi_{\lambda^{(m)}}(x_1,\ldots,x_m)
$$
By setting $x_1 = \ldots = x_n=1$, this shows that
\begin{equation}\label{Gibbs}
\sum_{\lambda^{(m)} \in \mathbb{GT}_m} \Lambda\left(\lambda^{(n)},\lambda^{(m)}\right) =1, \quad \text{ where } \Lambda\left(\lambda^{(n)},\lambda^{(m)}\right) =m\left(\lambda^{(n)},\lambda^{(m)}\right) \frac{\dim\lambda^{(m)}}{\dim\lambda^{(n)}}
\end{equation}

The branching rules allow for the representation $\pi_{\lambda}^{(q)}$ to be written explicitly; the exposition here comes from section 7.3.3 of \cite{KSBook}. More specifically, the carrier space $V_{\lambda}$ has a basis formed by successive restrictions to subalgebras $\mathcal{U}_q(\mathfrak{gl}_{n-1}),\ldots, \mathcal{U}_q(\mathfrak{gl}_1)$. By the branching rule, each $V_{\mu}$ appears in $V_{\lambda}$ exactly once. Since the irreducible representations of $\mathcal{U}_q(\mathfrak{gl}_1)$ are one--dimensional, we obtain a basis of $V_{\lambda}$ indexed by Gelfand--Tsetlin patterns, which are arrays of integers $m_{ij}$ satisfying
$$
m_{i,j+1} \geq m_{ij} \geq m_{i+1,j+1}, \quad 1 \leq i \leq j \leq n
$$
with $m_{i,n}=\lambda_i$. These can be visualized as
$$
\mathcal{M}=\left(\begin{array}{ccccccccc}
m_{1, n} & & m_{2, n} & & \cdots & \cdots & & m_{n,n} \\
 & m_{1, n-1} & & m_{2, n-1} & \cdots & & m_{n-1, n-1}   \\ & & \cdots & \cdots & \cdots & \cdots & & &  \\ & & & & & m_{11} & & & \end{array}\right)
$$
 Then $\pi^{(q)}_{\lambda}$ is given by (see Theorem 24 of section 7.3.3 of \cite{KSBook})
 $$
\begin{array}{c}\pi_{\lambda}^{(q)}\left(q^{E_{kk}}\right)|\mathcal{M}\rangle=q^{a_{k} }|\mathcal{M}\rangle, \quad a_{k}=\sum_{i=1}^{k} m_{i, k}, \quad 1 \leq k \leq n \\ \pi_{\lambda}^{(q)}\left(E_{k,k+1}\right)|\mathcal{M}\rangle=\sum_{j=1}^{k} A_{k}^{j}(q,\mathcal{M})\left|\mathcal{M}_{k}^{j}\right\rangle, \quad \pi_{\lambda}^{(q)}\left(E_{k+1,k}\right)|\mathcal{M}\rangle=\sum_{j=1}^{k} A_{k}^{j}\left(q,\mathcal{M}_{k}^{-j}\right)\left|\mathcal{M}_{k}^{-j}\right\rangle, 1 \leq k \leq n-1. \\ \end{array}
$$
Here, $\mathcal{M}_k^{\pm j}$ is the Gelfand--Tsetlin pattern obtained from $\mathcal{M}$ by replacing $m_{jk}$ with $m_{jk} \pm 1$, and $A_k^j(q,\mathcal{M})$ is the expression
$$
A_{k}^{j}(q,\mathcal{M})=\left(-\frac{\prod_{i=1}^{k+1}\left[l_{i, k+1}-l_{j, k}\right]_q \prod_{i=1}^{k-1}\left[l_{i, k-1}-l_{j, k}-1\right]_q}{\prod_{i \neq j}\left[l_{i, k}-l_{j, k}\right]_q\left[l_{i, k}-l_{j, k}-1\right]_q}\right)^{1 / 2},
$$
where $l_{ir}=m_{ir}-i$, the positive value of the square root is taken, and $[m]_q=(q^m-q^{-m})/(q-q^{-1})$. Multiplying the basis elements by appropriate factors gives an explicit formula for $\pi_{\lambda}^{(q)}$ for any complex value of $q$ which is not a root of unity. Note that the Gelfand--Tsetlin basis does not depend on $q$. Therefore, $\pi_{\lambda}^{(q)}$ and $\pi_{\lambda}^{(\tilde{q})}$ can be composed for different values of $q,\tilde{q}$. For example, 
$$
\pi_{\lambda}^{(q)}(E_{k,k+1}) \pi_{\lambda}^{(\tilde{q})}(E_{l,l+1}) \vert \mathcal{M}\rangle = \sum_{j=1}^k \sum_{j'=1}^l A_k^j(q, \mathcal{M}_{j'}^l) A_l^{j'}(\tilde{q},\mathcal{M}) \vert (\mathcal{M}_{j'}^l)^k_j\rangle.
$$
This explicit expression of $\pi_{\lambda}^{(q)},\pi_{\mu}^{(q)}$ also extends to external direct sums $V_{\lambda} \oplus_{\mathbb{C}} V_{\mu}$, by
$$
( \pi_{\lambda}^{(q)}(u) \oplus \pi_{\mu}^{(q)}(u))\left(\vert \mathcal{M}\rangle \oplus \vert \mathcal{M}' \rangle\right)= \pi_{\lambda}^{(q)}(u) \vert \mathcal{M}\rangle \oplus \pi_{\mu}^{(q)}(u) \vert \mathcal{M}'\rangle.
$$
Again, $ \pi_{\lambda}^{(q)}(u) \oplus \pi_{\mu}^{(q)}(u)$ can be composed with  $\pi_{\lambda}^{(\tilde{q})}(u) \oplus \pi_{\mu}^{(\tilde{q})}(u)$ for $q \neq \tilde{q}$.

\begin{remark}
If $\iota_q: W \rightarrow V_{\lambda} \oplus V_{\mu}$ is an isomorphism of $\mathcal{U}_q(\mathfrak{gl}_n)$--modules, then $\iota_q$ in general will not intertwine with $\pi_{\lambda}^{(\tilde{q})} \oplus \pi_{\mu}^{(\tilde{q})}$ unless $q=\tilde{q}$. However, this does not present an issue for our purposes in this paper.
\end{remark}

Because $\mathcal{U}_q(\mathfrak{gl}_n)$ also carries the structure of a $^*$--algebra, the notion of a $^*$--representation is also needed. A representation $\pi$ of a $^*$--algebra $U$ on a vector space $V$ with a scalar product is called a $^*$--representation if
\begin{equation}\label{Star}
\langle \pi(a)v,w\rangle = \langle v, \pi(a^*)w\rangle
\end{equation}
for all $a\in U$ and $v,w\in V$. Every finite--dimensional irreducible representation of $\mathcal{U}_q(\mathfrak{gl}_n)$ also carries the structure of a $*$--representation; see Chapter 7 of \cite{KSBook}. 

{\color{black} 
\begin{remark} We note that the irreducible representations of $\mathcal{U}_q(\mathfrak{gl}_n)$, as well as their weight space decompositions, can be taken independently of $q$. This fact will be implicitly used later when we multiply difference operators corresponding to different values of $q$. As mentioned in the introduction, in the $q\rightarrow 1$ limit this fact also explains why the Gaussian fields $\mathfrak{G}$ and $\mathcal{G}$ match along space--like paths.
\end{remark}
}

{\color{black}
\subsubsection{Dual Hopf Algebra}
Recall some general definitions; see e.g. \cite{KoS}. If $A$ is a complex Hopf algebra, then the matrix elements of finite--dimensional $A$--modules form a Hopf subalgebra $A^*$ in the dual algebra $\mathrm{Hom}_{\mathbb{C}}(A,\mathbb{C})$. For any Hopf algebra, the right and left regular representations are the actions of $A$ on $A^*$ given by
\begin{align*}
\mathcal{R}: f \mapsto & \langle \mathrm{id} \otimes a, \Delta(f) \rangle, \\
\mathcal{L}: f \mapsto & \langle S^{-1}(a)\otimes \mathrm{id}, \Delta(f)\rangle.
\end{align*}

For nonzero real values of $q$, the quantum group $\mathcal{U}_q(\mathfrak{gl}_n)$ has a dual Hopf $*$--algebra, which we will denote $\mathbb{C}[GL(n)]_q$, following \cite{KoS}. This algebra is variously called {the quantized algebra of regular functions}, or the {quantum coordinate algebra}, or {the algebra of regular functions on the quantum group}. It can be viewed as a $\Uq \otimes \Uq$--bimodule under the action $\mathcal{L} \otimes \mathcal{R}$. As a bi--module, it has the Peter--Weyl decomposition
$$
\mathbb{C}[GL(n)]_q \cong \bigoplus_{\lambda \in \mathbb{GT}_n} V_{\lambda}^{\oplus \dim \lambda}
$$
with the isomorphism given by
$$
c_{l,v}^{V_{\lambda}} \mapsto l \otimes v,
$$
where $c_{l,v}^{V_{\lambda}}$ denotes the matrix element of type $l \in V_{\lambda}^*$ and $v \in V_{\lambda}$. (Recall that $V_{\lambda}^* \otimes V_{\lambda} $ is isomorphic to the space of endomorphisms $ \mathrm{End}(V_{\lambda},V_{\lambda})$).

In the classical case, the Peter--Weyl decomposition gives
$$
L^2(U(n)) = \bigoplus_{\lambda \in \mathbb{GT}_n} V_{\lambda}^{\oplus \dim \lambda}.
$$
Therefore, we can define the von Neumann algebra $M$ of operators in $\mathrm{End}_{\mathbb{C}}(L^2(U(n)))$ which preserve each summand in the Peter--Weyl decomposition. (This is the same definition that had been previously used in \cite{K}). Since $L^2(U(n))$ is a $\Uq \otimes \Uq$--bimodule for all nonzero real values of $q$, we have morphisms
$$
\mathcal{L} \otimes \mathcal{R} : \Uq \otimes \Uq \rightarrow M.
$$
Because we have two morphisms $\Uq \rightarrow \Uq \otimes \Uq$ defined by $u\mapsto u\otimes 1$ and $u\mapsto 1\otimes u$, these can be composed with $\mathcal{L}\otimes \mathcal{R}$ to obtains two morphisms from $\Uq$ to $M$. Equivalently, we have a left action and a right action of $\Uq$ on its dual $\mathbb{C}[GL(n)]_q$. Let $\pi^{(q)}(u) \in M$ denote the right action of $u\in \Uq$. In other words, $\pi^{(q)}(u)$ acts on $L^2(U(n))$ as 
\begin{equation}\label{New}
\pi^{(q)}(u) =  \bigoplus_{\lambda \in \mathbb{GT}_n} \left(\pi^{(q)}_{\lambda}\right)^{\oplus \dim \lambda}.
\end{equation}

Note that because every finite--dimensional representation of $\Uq$ is also a $*$--representation, we have that
\begin{equation}\label{Stars}
\pi^{(q)}(u^*) = \pi^{(q)}(u)^* .
\end{equation}
}

For $m<n$, the right action of $u \in \mathcal{U}_q(\mathfrak{gl}_m)$ on $\mathbb{C}[GL(n)]_q$ is well--defined: because $\mathcal{U}_q(\mathfrak{gl}_m) \subset \mathcal{U}_q(\mathfrak{gl}_n)$, the representation $V_{\lambda}$ to can be restricted to $\mathcal{U}_q(\mathfrak{gl}_m)$. Thus, $u \in \mathcal{U}_q(\mathfrak{gl}_m) $ has a well--defined right action on $V_{\lambda}$, so we have a morphism from $\mathcal{U}_q(\mathfrak{gl}_m) $ to $M$.

\section{Quantum Random Walks on Quantum Groups}\label{QRWonQG}
Before defining the quantum random walk, first define is the states. For $t\geq 0$ let $\chi^t(x_1,\ldots,x_n)$ denote
$$
\chi^t(x_1,\ldots,x_n) = e^{t(x_1-1 + \cdots +x_n-1)} = e^{-tn} e^{t(x_1+\cdots + x_n)}.
$$
By a result of \cite{BF,BK} (see also section \ref{Ctcp} below),
$$
\chi^t(x_1,\ldots,x_n) = \sum_{\lambda \in \mathbb{GT}_n} P(\lambda,t) \frac{s_{\lambda}(x_1,\ldots,x_n)}{s_{\lambda}(1,\ldots,1)}
$$
for some non--negative coefficients $P(\lambda,t)$. Since $s_{\lambda}$ is the character of the representation $V_{\lambda}$, this motivates the following definition. Since $D\in M$ preserves each summand in the Peter--Weyl decomposition, we can define a state on $M$ by
$$
\langle D \rangle_t = \sum_{\lambda \in \mathbb{GT}_n} P(\lambda,t) \frac{\mathrm{Tr}|_{V_{\lambda}}(D)}{\dim V_{\lambda}}.
$$
We can also define a state on $\Uq$  by 
\begin{equation}\label{Du}
\langle u \rangle_{t} =  \langle\pi^{(q)}(u) \rangle_t
\end{equation}
The positivity of these states follows immediately from the fact that $\mathrm{Tr}(AA^*)\geq 0$ for all matrices $A$, and that $\pi^{(1)}$ is a $^*$--morphism.

Now that the states have been defined, we define the non--commutative random walk.  Fix times $t_1<t_2<\ldots$.  Let $\mathcal{W}$ be the direct limit
$
\varinjlim M^{\otimes k},
$
where $M^{\otimes k} \rightarrow M^{\otimes k+1}$ maps $w_1 \otimes \cdots \otimes w_k \mapsto w_1 \otimes \cdots \otimes w_k \otimes 1$. Define the state $\omega:=\state{\cdot}_{t_1} \otimes \state{\cdot}_{t_2-t_1} \otimes \cdots$ on $\mathcal{W}$. 
For $k\geq 1$ define the map $j_{k}:\Uq \rightarrow \mathcal{W}$ to be the map 
$$
j_{k}(u) = [\pi^{(q)}]^{\otimes k}(\Delta^{(k-1)}u) \otimes \mathrm{Id}^{\otimes \infty} =  \sum_{(u)} \pi^{(q)}{(u_{1})} \otimes \cdots \otimes \pi^{(q)}{(u_k)} \otimes \mathrm{Id}^{\otimes\infty}.
$$ 
and let $\mathcal{W}_k$ be the subalgebra generated by the images of $j_{1},\ldots,j_{k}$. Let $P^{}_t$ be the non--commutative Markov operator on $\Uq$ defined by $P_t=(\mathrm{id} \otimes \langle \cdot \rangle_t)\circ\Delta$.  

Now fix real values $q_1,q_2,\ldots$; we will now define states and non--commutative random walk that combine these different values of $q$. Before giving the rigorous definitions, first let us provide some intuition. In the case where all $q_j=q$, we know that since $\pi^{(q)}$ is a morphism, \eqref{Du} implies that for $u^{(1)},\ldots,u^{(r)} \in \Uq$,
$$
\state{ m_r(u^{(1)} \otimes \cdots u^{(r)})  }_t = \state{  \pi^{(q_1)}(u^{(1)}) \circ \cdots \circ \pi^{(q_r)}({u^{(r)}})  }_{t} ,
$$
where $m_r$ is the multiplication map. Similarly, 
$$
j_k(m_r(u^{(1)} \otimes \cdots u^{(r)})) = \sum_{u^{(1)},\ldots, u^{(r)}} [\pi^{(q)}(u^{(1)}_1) \circ \cdots \circ \pi^{(q)}(u^{(r)}_1)] \otimes \cdots \otimes [\pi^{(q)}(u^{(1)}_k) \circ \cdots \circ \pi^{(q)}(u^{(r)}_k)].
$$
and
$$
P_t(m_r(u^{(1)} \otimes \cdots \otimes  u^{(r)})) =  \sum_{u^{(1)},\ldots, u^{(r)}} \state{m_r(u^{(1)}_2 \otimes \cdots \otimes u^{(r)}_2)}_t m_r(u^{(1)}_1 \otimes \cdots \otimes u^{(r)}_2).
$$
If the values of $q_j$ differ, then there is no map $m_r$; however, it turns out that formally removing $m_r$ does not form an impediment to constructing the non--commutative random walk. 

Before continuing, let us first prove a lemma:
\begin{lemma}
Let $\lambda \in \mathbb{GT}_n$ and $n_1,n_2\leq n$. If $u_1$ and $u_2$ are elements in the centers of $\mathcal{U}_{q_1}(\mathfrak{gl}_{n_1}), \mathcal{U}_{q_2}(\mathfrak{gl}_{n_2})$ respectively, then
\begin{equation*}
\pi_{\lambda}^{(q_1)}(u^{(1)}) \circ \pi_{\lambda}^{(q_2)}(u^{(2)}) = \pi_{\lambda}^{(q_2)}(u^{(2)}) \circ \pi_{\lambda}^{(q_1)}(u^{(1)}) ).
\end{equation*}
Furthermore, 
\begin{equation}\label{Comm}
\pi_{}^{(q_1)}(u^{(1)}) \circ \pi_{}^{(q_2)}(u^{(2)}) = \pi_{}^{(q_2)}(u^{(2)}) \circ \pi_{}^{(q_1)}(u^{(1)}) )
\end{equation}
\end{lemma}
\begin{proof}
The fact that $\pi_{\lambda}^{(q_1)}(u^{(1)})$ can be composed with $\pi_{\lambda}^{(q_2)}(u^{(2)})$ follows from the discussion in section \ref{3.2.2} about the Gelfand--Tsetlin basis. Since diagonal matrices commute with each other, it suffices to show that $\pi_{\lambda}^{(q_i)}(u^{(i)})$ acts as a diagonal matrix with respect to the Gelfand--Tsetlin basis. Without loss of generality take $i=1$. By the explicit expression for $\pi_{\lambda}^{(q)}$ in section \ref{3.2.2}, the action of any element of $\mathcal{U}_{q_1}(\mathfrak{gl}_{n_1})$ on $\vert \mathcal{M}\rangle$ only depends on the bottom $n_1$ rows of $\mathcal{M}$. Combined with the fact that any central element of $\mathcal{U}_{q_1}(\mathfrak{gl}_{n_1})$ acts as a constant on any irreducible representation of $\mathcal{U}_{q_1}(\mathfrak{gl}_{n_1})$, this means that $\pi_{\lambda}^{(q_1)}(u^{(1)}) \vert \mathcal{M} \rangle = C(q_1,u^{(1)},m_{1,n_1},\ldots,m_{n_1,n_1})\vert \mathcal{M} \rangle$ for all $\mathcal{M}$. Here, $C(q_1,u^{(1)},m_{1,n_1},\ldots,m_{n_1,n_1})$ is a constant that only depends on the listed variables. In particular, $\pi_{\lambda}^{(q_1)}(u^{(1)})$ acts as a diagonal matrix with respect to the Gelfand--Tsetlin basis, as needed.

The second part of the lemma follows from \eqref{New} and the first part of the lemma.

\end{proof}
   Given $u^{(1)} \in \mathcal{U}_{q_1}(\mathfrak{gl}_{n_1}),\ldots,u^{(r)} \in \mathcal{U}_{q_r}(\mathfrak{gl}_{n_r})$, {\color{black} define}
\begin{equation}\label{Defn}
\state{ u^{(1)} \ro u^{(2)} \ro \cdots \ro u^{(r)}}_{t} = \state{  \pi^{(q_1)}(u^{(1)}) \circ \cdots \circ \pi^{(q_r)}({u^{(r)}})  }_{t}.
\end{equation}
The order of the composition on the right--hand--side matters in general; however, if every $u_i$ is central in $\mathcal{U}_{q_i}(\mathfrak{gl}_{n_i})$, then by \eqref{Comm} the order does not matter. While \eqref{Defn} does not define a state on the entire algebra, it does define a state on $Z(\mathcal{U}_{q_1}(\mathfrak{gl}_{n_1})) \otimes \cdots \otimes Z(\mathcal{U}_{q_r}(\mathfrak{gl}_{n_r}))$ -- to see this, note that 
\begin{align*}
\state{ (u^{(1)} \ro u^{(2)} \ro \cdots \ro u^{(r)})(u^{(1)} \ro u^{(2)} \ro \cdots \ro u^{(r)})^*}_{t}  &\stackrel{\eqref{Stars}}{=} \state{\pi^{(q_1)}(u^{(1)})\pi^{(q_1)}(u^{(1)})^*  \circ \cdots \circ \pi^{(q_r)}({u^{(r)}}) \pi^{(q_r)}({u^{(r)}}) ^*}_{t} \\
& \stackrel{\eqref{Comm}}{=}  \state{ ( \pi^{(q_1)}(u^{(1)})  \cdots  \pi^{(q_r)}({u^{(r)}}))( \pi^{(q_1)}(u^{(1)})  \cdots  \pi^{(q_r)}({u^{(r)}}))^* }_t .
\end{align*}

Similarly, let $j_k^{( r)}$ be the map from $\mathcal{U}_{q_1}(\mathfrak{gl}_{n_1}) \otimes \cdots \otimes \mathcal{U}_{q_r}(\mathfrak{gl}_{n_r})$ to $\mathcal{W}$ defined by
\begin{multline*}
v^{(1)} \otimes \cdots \otimes v^{(r)} \mapsto \sum_{v^{(1)}} \cdots \sum_{v^{(r)}}   \left[ \pi^{(q_1)} (v_1^{(1)}) \circ \pi^{(q_2)}( v_1^{(2)} )\circ \cdots \circ \pi^{(q_r)}(v_1^{(r)} ) \right]\\
 \otimes \left[ \pi^{(q_1)} ( v_2^{(1)}) \circ \pi^{(q_2)}( v_2^{(2)} ) \circ \cdots \circ \pi^{(q_r)}( v_2^{(r)} )\right]\\
 \otimes \cdots \otimes \left[  \pi^{(q_1)} ( v_k^{(1)}) \circ \pi^{(q_2)}( v_k^{(2)}) \circ \cdots \circ \pi^{(q_r)}(v_k^{(r)} ) \right]
\end{multline*}
Let $\mathcal{W}^{(r)}_k$ be the subalgebra generated by the images of $j^{(r)}_{1},\ldots,j^{(r)}_{k}$. Let $\omega^{(r)}$ be the state $\state{\cdot}_{t_1} \otimes \state{\cdot}_{t_2-t_1} \otimes \cdots $. Let $P_t^{( r)}$ be the operator on $\mathcal{U}_{q_1}(\mathfrak{gl}_{n_1}) \otimes \cdots \otimes \mathcal{U}_{q_r}(\mathfrak{gl}_{n_r})$ defined by 
$$
v^{(1)} \otimes \cdots \otimes v^{(r)} \mapsto \sum_{v^{(1)}} \cdots \sum_{v^{(r)}} \state{v^{(1)}_2 \otimes \cdots \otimes v^{(r)}_2}_t v^{(1)}_1 \otimes \cdots \otimes v^{(r)}_1
$$



\begin{theorem}\label{QRW}
Assume that {\color{black}  $q_1,\ldots,q_r$ are real and nonzero. Then

(0)As maps from $Z(\mathcal{U}_{q}(\mathfrak{gl}_{N_1})) \otimes \cdots \otimes Z(\mathcal{U}_{q}(\mathfrak{gl}_{N_r})) $ to $\mathcal{U}_q(\mathfrak{gl}_N)$, where $N=\mathrm{max}(N_1,\ldots,N_r)$, there is the relation  $P_t \circ m_r = m_r \circ P_t^{( r)}.$}

(1a) The maps $(j_n)$ are related to $P_t$ by 
$$
\omega\left( w j_n(X)  \right) = \omega\left(w j_{n-1}(P_{t_n-t_{n-1}}X) \right), \quad X \in \mathcal{U}_q(\mathfrak{gl}_N), \quad w \in \mathcal{W}_{n-1}.
$$

{\color{black}
(1b) The maps $(j_n^{(r)})$ are related to $P_t^{(r)}$ by

$$
\omega^{(r)}\left(w  j_n^{(r)}(X)  \right) = \omega\left(w j_{n-1}^{(r)}(P_{t_n-t_{n-1}}^{(r)}X) \right), \quad X \in \mathcal{U}_{q_1}(\mathfrak{gl}_{N_1}) \otimes \cdots \otimes \mathcal{U}_{q_r}(\mathfrak{gl}_{N_r}), \quad w \in \mathcal{W}^{(r)}_{n-1}.
$$
}

(2a) The non--commutative Markov operators preserve the states in the sense that 
$$
\langle P_t u\rangle_{\chi_{\lambda}} = \langle u_{1} \rangle_{\chi_{\lambda}} \langle u_{2}\rangle_t = \langle u\rangle_{\chi_{\lambda}\chi^t}
$$
 and satisfy the semi--group property $P_{t+s} = P_t \circ P_s$. 
 
(2b) Similarly,
$$
\langle P_t^{(r)} (u^{(1)} \otimes \cdots \otimes u^{(r)} )\rangle_{\chi_{\lambda}} = \langle u^{(1)}_{1} \otimes \cdots \otimes u^{(r)}_1 \rangle_{\chi_{\lambda}} \langle u^{(1)}_{2} \otimes \cdots \otimes u^{(r)}_2 \rangle_t = \langle u^{(1)} \otimes \cdots \otimes u^{(r)} \rangle_{\chi_{\lambda}\chi^t}
$$
and there is the semigroup property $P_t^{(r)} \circ P_s^{(r)}= P_{t+s}^{(r)}$.

(3a) The pullback of $\omega$ under $j_n$ is the state $\langle \cdot \rangle$ on $ \mathcal{U}_q(\mathfrak{gl}_N)$, i.e. $\langle X\rangle_{t_n} = \omega(j_n(X))$.

{\color{black} (3b)  Similarly, for $X \in \mathcal{U}_{q_1}(\mathfrak{gl}_{N_1}) \otimes \cdots \otimes \mathcal{U}_{q_r}(\mathfrak{gl}_{N_r})$ we have 
$
\langle X\rangle_{t_n} = \omega^{(r)}(j_n^{(r)}(X)).
$
}

(4a) If $X,Y\in \mathcal{U}_q(\mathfrak{gl}_N)$, then for $n \leq m$ we have
$$
\omega\left(j_n(X)j_m(Y)\right) = \langle X \cdot P_{t_m-t_n}Y\rangle_{t_n}  
$$

{\color{black}
(4b) If $X^{(1)},\ldots,X^{(j)} \in \mathcal{U}_{q_1}(\mathfrak{gl}_{N_1})$ and $Y^{(1)},\ldots,Y^{(r)} \in \mathcal{U}_{q_r}(\mathfrak{gl}_{N_r})$, then for $n\leq m$ 
\begin{multline*}
\omega^{(r)}\left(j_n^{(r)}(X^{(1)} \otimes \cdots \otimes X^{(r)} )j_m^{(r)}(Y^{(1)} \otimes \cdots \otimes Y^{(r)})\right) \\
= \langle (X^{(1)} \otimes \cdots \otimes X^{(r)} ) \cdot P^{(r)}_{t_m-t_n}(Y^{(1)} \otimes \cdots \otimes Y^{(r)})\rangle_{t_n}  
\end{multline*}

}

(5a) Suppose $n_1 \leq \cdots \leq n_{k-1} \leq n_k$. Let $X_i \in \mathcal{U}_q(\mathfrak{gl}_{n_i})$  for $1 \leq i \leq k$. Then
$$
\omega\left(j_{n_1}(X_1)\cdots j_{{n_{k-1}}}(X_{k-1}) j_{n_{k}}(X_k)\right) =  \left\langle X_1 {\color{black} \cdot P^{}_{t_{n_{2}}-t_{n_{1}}} }\left( X_2 \cdot P^{ }_{t_{n_3}-t_{n_2}} \left( X_3 \cdot \ldots \cdot P^{ }_{t_{n_k}-t_{n_{k-1}}} X_r \right)\right)\right\rangle_{t_{n_1}} 
$$

(5b) Suppose $n_1 \leq \cdots \leq n_{r-1} \leq n_r$. Let $X_i \in U_{q_1}(\mathfrak{gl}_{n_1}) \otimes \cdots \otimes U_{q_r}(\mathfrak{gl}_{n_r})$ for $1 \leq i \leq k$. Then
$$
\omega^{(r)}\left(j_{n_1}^{(r)}(X_1)\cdots j_{{n_{k-1}}}^{(r)}(X_{k-1}) j_{n_{k}}^{(r)}(X_k)\right) =  \left\langle X_1 {\color{black} \cdot P^{(r)}_{t_{n_{2}}-t_{n_{1}}} }\left( X_2 \cdot P^{(r) }_{t_{n_3}-t_{n_2}} \left( X_3 \cdot \ldots \cdot P^{ (r) }_{t_{n_k}-t_{n_{k-1}}} X_r \right)\right)\right\rangle_{t_{n_1}} 
$$
\end{theorem}
\begin{proof}
{\color{black} (0) Because of the embedding $\mathcal{U}_q(\mathfrak{gl}_{N_j}) \subseteq \mathcal{U}_q(\mathfrak{gl}_N)$, we can take $v^{(1)} \otimes \cdots \otimes v^{(r)} \in Z(\mathcal{U}_q(\mathfrak{gl}_N)^{\otimes r} )$.  Then $P_t \circ m_r$ applied to this element is
\begin{align*}
(\mathrm{id} \otimes \state{\cdot}_t ) ( \Delta( v^{(1)} \cdots v^{(r)} ) )&= (\mathrm{id} \otimes \state{\cdot}_t )\sum_{(v^{(1)})} \cdots  \sum_{(v^{(r)})} (v^{(1)}_1 \cdots v^{(r)}_1 \otimes v^{(1)}_2 \cdots v^{(r)}_2) \\
&= \state{v^{(1)}_2 \cdots v^{(r)}_2}_t \sum_{(v^{(1)})} \cdots  \sum_{(v^{(r)})} v^{(1)}_1 \cdots v^{(r)}_1,
\end{align*}
which equals
$$
\state{v^{(1)}_2 \otimes \cdots \otimes  v^{(r)}_2}_t \sum_{(v^{(1)})} \cdots  \sum_{(v^{(r)})} m_r( v^{(1)}_1 \otimes \cdots \otimes v^{(r)}_1 )
$$
By equation \eqref{Defn}, this equals $m_r \circ P_t^{(r)}$ applied to the same element.

}

(1a) The proof is similar to Theorem 4.1(1) from \cite{K}, which was itself based off of Proposition 3.1 from \cite{CD}. The left--hand--side is
\begin{align*}
\omega\left( w (\pi^{\otimes n-1}\otimes \pi)\Delta X \right) &= \sum_{(X)} \omega\left( w \pi^{\otimes n-1}(X_1) \otimes \pi(X_2) \right) \\
&= \sum_{(X)} \omega\left(w \pi^{\otimes n-1}(X_1)\right) \state{X_2}_{t_n-t_{n-1}}
\end{align*}
The right--hand--side is
$$
\sum_{(X)} \omega\left( w j_{n-1}\left(\state{X_2}_{t_n-t_{n-1}} X_1\right) \right) = \sum_{(X)} \omega\left(w j_{n-1}(X_1) \right) \state{X_2}_{t_n-t_{n-1}}.
$$
So the two sides are equal. Note that this argument did not assume that $j_n$ is a homomorphism. 

{\color{black}
(1b) We argue similarly as in (1a). The right--hand--side is 
\begin{multline*}
\sum_{(X^{(1)}),\cdots, (X^{(r)})} \omega\left(w j_{n-1}^{(r)}\left(\state{X^{(1)}_2 \otimes \cdots \otimes X^{(r)}_2}_{t_n-t_{n-1}} X^{(1)}_1 \otimes \cdots \otimes X^{(r)}_1\right)\right)\\
=\sum_{(X^{(1)}),\cdots, (X^{(r)})} \omega\left( w j_{n-1}^{(r)}\left( X^{(1)}_1 \otimes \cdots \otimes X^{(r)}_1\right)w\right) \state{X^{(1)}_2 \otimes \cdots \otimes X^{(r)}_2}_{t_n-t_{n-1}}
\end{multline*}
By the definition of $\omega^{(r)}$ and $j_{n}^{(r)}$, this equals the left--hand--side.
}

(2a)
First we show
$$
\sum_{(u)} \langle u_1 \rangle_{\chi_{\lambda}} \langle u_2\rangle_t = \langle u\rangle_{\chi_{\lambda}\chi^t}
$$
Because $\{\chi_{\lambda}:\lambda \in \mathbb{GT}_n\}$ is a linear basis for the space of symmetric functions, it suffices to show that
$$
\sum_{(u)}  \mathrm{Tr}|_{V_{\lambda}} \left( u_1 \right) \mathrm{Tr}|_{V_{\mu}} \left(u_2 \right) =\mathrm{Tr}|_{V_{\lambda}\otimes V_{\mu}} \left(u \right) .
$$
But this is true because the co--product is what defines the action on tensor powers of representations, as in \eqref{Tense}.

Now, by the co--associativity property,
$$
(\mathrm{id} \otimes \Delta) \circ \Delta(u) = (\Delta \otimes \mathrm{id}) \circ \Delta(u) \Longrightarrow \sum_{(u)} \sum_{(u_2)} u_{1} \otimes u_{21}\otimes u_{22} = \sum_{(u)} \sum_{(u_1)} u_{11} \otimes u_{12} \otimes u_{2}
$$
Therefore
\begin{align*}
P_t \circ P_s (u) &= \sum_{(u)} P_t \left( u_1 \state{u_2}_s\right) = \sum_{(u), (u_1)} u_{11} \state{ u_{12}}_t \state{u_2}_s \\
&=  \sum_{(u)} \sum_{(u_1)} u_{11} \left(\mathrm{id} \otimes \state{\cdot}_t \otimes \state{\cdot}_s \right)\left( u_{11} \otimes u_{12} \otimes u_2 \right) \\
&= \sum_{(u)} \sum_{(u_2)}  u_1 \state{ u_{21}}_t \state{u_{22}}_s = \sum_{(u)} u_1 \state{ u_2}_{t+s} = P_{t+s}(u)
\end{align*}

(2b) We argue similarly as in (2a). The first part of the statement has an identical proof. For the second part:
\begin{align*}
&  (P^{(r)}_t \circ P^{(r)}_s )(v^{(1)} \otimes \cdots \otimes v^{(r)}) \\
\quad &=  \sum_{(v^{(1)}),\ldots,(v^{(r)})} \state{v^{(1)}_2 \otimes \cdots \otimes v^{(r)}_2}_s P_t^{(r)}\left(v^{(1)}_1 \otimes \cdots \otimes v^{(r)}_1\right) \\
\quad &=  \sum_{ { (v^{(1)}),\ldots,(v^{(r)}) }} \sum_{(v^{(1)}_1),\ldots,(v^{(r)}_1)}   \state{v^{(1)}_2 \otimes \cdots \otimes v^{(r)}_2}_s  \state{ v^{(1)}_{12} \otimes \cdots \otimes v^{(r)}_{12} }_t  v^{(1)}_{11} \otimes \cdots \otimes v^{(r)}_{11}
\end{align*}
Again by co--associativity, this equals
$$
\sum_{ { (v^{(1)}),\ldots,(v^{(r)}) }} \sum_{(v^{(1)}_2),\ldots,(v^{(r)}_2)}   \state{v^{(1)}_{22} \otimes \cdots \otimes v^{(r)}_{22} }_s  \state{ v^{(1)}_{21} \otimes \cdots \otimes v^{(r)}_{21} }_t  v^{(1)}_{1} \otimes \cdots \otimes v^{(r)}_{1}.
$$
Thus, as in the proof of (2a), it suffices to show that 
$$
 \sum_{(v^{(1)}_2),\ldots,(v^{(r)}_2)}  \state{v^{(1)}_{22} \otimes \cdots \otimes v^{(r)}_{22} }_s  \state{ v^{(1)}_{21} \otimes \cdots \otimes v^{(r)}_{21} }_t   =  \state{v^{(1)}_2 \otimes \cdots \otimes v^{(r)}_2}_{s+t} 
$$
By equation \eqref{Defn}, it is equivalent to show that  
\begin{multline}
 \sum_{(v^{(1)}_2),\ldots,(v^{(r)}_2)}  \state{ \pi^{(q_1)}({v^{(1)}_{22}}) \circ \cdots \circ \pi^{(q_r)}({v^{(r)}_{22}})  }_s  \state{ \pi^{(q_1)} ( {v^{(1)}_{21} )} \circ \cdots \circ \pi^{(q_r)}( {v^{(r)}_{21}}) }_t  \\
 = \state{ \pi^{(q_1)}  ( v_2^{(1)}  ) \circ \cdots \pi^{(q_r)}(v_r^{(1)} )  }_{s+t}.
\end{multline}
Because $\{\chi_{\lambda}:\lambda \in \mathbb{GT}_n\}$ is a linear basis for the space of symmetric functions,  it suffices to show that
\begin{multline*}
 \sum_{(v^{(1)}_2),\ldots,(v^{(r)}_2)}  \mathrm{Tr}\Big|_{V_{\lambda}}\left(  \pi^{(q_1)}_{\lambda}({v_{22}^{(1)}})   \circ \cdots \circ  \pi^{(q_r)}_{\lambda}({v_{22}^{(r)}})   \right) \mathrm{Tr}\Big|_{V_{\mu}}\left(  \pi^{(q_1)}_{\mu}({v_{21}^{(1)}})   \circ \cdots \circ  \pi^{(q_r)}_{\mu}({v_{21}^{(r)}})   \right)\\
= \mathrm{Tr}\Big|_{V_{\lambda}\otimes V_{\mu}}\left(  \pi^{(q_1)}_{\lambda \otimes \mu }(v_{2}^{(1)})   \circ \cdots \circ  \pi^{(q_r)}_{\lambda \otimes \mu}(v_{2}^{(r)})   \right)
\end{multline*}
Letting $\{l_i: 1 \leq i \leq \dim \lambda\}$ and $\{m_j: 1 \leq j \leq \dim \mu\}$, it is equivalent to show that
\begin{multline*}
 \sum_{(v^{(1)}_2),\ldots,(v^{(r)}_2)} \sum_{i,j}\langle l_i |  \pi^{(q_1)}_{\lambda}({v_{22}^{(1)}})   \circ \cdots \circ  \pi^{(q_r)}_{\lambda}({v_{22}^{(r)}})  | l_i\rangle \langle m_j | \pi^{(q_1)}_{\mu}({v_{21}^{(1)}})   \circ \cdots \circ  \pi^{(q_r)}_{\mu}({v_{21}^{(r)}})  | m_j \rangle \\
= \sum_{i ,j} \langle l_i \otimes m_j \vert  \pi^{(q_1)}_{\lambda \otimes \mu }(v_{2}^{(1)})   \circ \cdots \circ  \pi^{(q_r)}_{\lambda \otimes \mu}(v_{2}^{(r)})   \vert l_i \otimes m_j \rangle.
\end{multline*}
Therefore, it is sufficient to show that
\begin{multline*}
  \sum_{(v^{(1)}_2),\ldots,(v^{(r)}_2)}  \pi^{(q_1)}_{\lambda}({v_{22}^{(1)}})    \cdots   \pi^{(q_r)}_{\lambda}({v_{22}^{(r)}})  | l_i\rangle \otimes \pi^{(q_1)}_{\mu}({v_{21}^{(1)}})    \cdots   \pi^{(q_r)}_{\mu}({v_{21}^{(r)}})  | m_j \rangle \\
= \pi^{(q_1)}_{\lambda \otimes \mu }(v_{2}^{(1)})    \cdots   \pi^{(q_r)}_{\lambda \otimes \mu}(v_{2}^{(r)})   \vert l_i \otimes m_j \rangle.
\end{multline*}
But this is true by \eqref{Tense}, so we can move on to (3).

(3a) By repeatedly applying (1a) with $w=1$, we have
$$
\omega\left( j_n(X)\right) = \omega\left( j_{n-1}(P_{t_n-t_{n-1}}X)\right) = \ldots = \omega\left( j_1(P_{t_2-t_1}\circ \cdots \circ P_{t_n-t_{n-1}}(X))\right)
$$
By the definition of $\omega$ and $j_1$, and by applying (2a) this equals 
$$
\state {P_{t_2-t_1}\circ \cdots \circ P_{t_n-t_{n-1}}(X)) }_{t_1} = \langle P_{t_n-t_{n-1}}(X)\rangle_{t_1} = \langle X\rangle_{t_n}
$$

{\color{black} (3b) The proof is identical to (3a), but uses (1b) and (2b) in place of (1a) and (2a).
}

(4a) By repeated applications of (1a), 
$$
\omega(j_n(X)j_m(Y)) = \omega\left(j_n(X)j_n(P_{t_{n+1}-t_n} \circ \cdots \circ P_{t_m-t_{m-1}}(Y)) \right) 
$$ 
And then by (2a), 
$$
\omega(j_n(X)j_m(Y))  = \omega\left(j_n(X)j_n(P_{t_m-t_n}(Y)) \right).
$$
Because $j_n$ is a morphism, using (3a) finishes the proof. 

{\color{black}
(4b) The proof is identical to the proof of (4a): one uses (2b) and (1b) instead of (2a) and (1a).
}

(5a) By repeated applications of (4a),
\begin{align*}
&\omega\left(j_{n_1}(X_1)\cdots j_{{n_{r-1}}}(X_{r-1}) j_{n_{r}}(Y_r)\right) \\
&= \omega\left(j_{n_1}(X_1)\cdots j_{{n_{r-1}}}(X_{r-1} \cdot P_{t_{n_r}-t_{n_{r-1}}}Y_r)  \right) \\
&= \omega\left(j_{n_1}(X_1)\cdots j_{{n_{r-2}}}(X_{r-2} \cdot ( P_{t_{n_{r-1}}-t_{n_{r-2}}}(X_{r-1} \cdot P_{t_{n_r}-t_{n_{r-1}}}Y_r)) )\right)\\
&= \cdots \\
&= \omega \left( j_{n_1}\left(X_1 P_{t_{n_{2}}-t_{n_{1}}}\left( X_2 P_{t_{n_3}-t_{n_2}} \left( X_3 \cdots P_{t_{n_r}-t_{n_{r-1}}} Y_r\right) \right)\right)\right)
\end{align*}
By (3a), this equals
$$
\left\langle X_1 P_{t_{n_{2}}-t_{n_{1}}}\left( X_2 P_{t_{n_3}-t_{n_2}} \left( X_3 \cdots P_{t_{n_r}-t_{n_{r-1}}} Y_r \right)\right)\right\rangle_{t_{n_1}} ,
$$
as needed.

(5b) The proof is identical to the proof of (5a): one uses (4b) and (3b) instead of (4a) and (3a).
\end{proof}

\begin{remark}\label{QT}
Note that in the $q=1$ case, the non--commutative Markov operator $P_t$ preserves the center (see Theorem 4.1(5) of \cite{K} or Proposition 4.3 of \cite{CD}). We will see below that this is not true for general $q$. However, one could use the \textit{quantum trace}
$$
\state{ u}_{\chi_V}^{(q)} = \mathrm{Tr}\vert_V \left( u q^{-2\rho}\right)
$$
where
$$
2\rho = (n-1)E_{11} + (n-3)E_{22} + \cdots + (1-n)E_{nn}.
$$
Then
$$
\state{ uv}^{(q)} = \mathrm{Tr}\left( uvq^{-2\rho}\right) = \mathrm{Tr}\left( v \cdot q^{-2\rho} u \right) = \state{ v\cdot q^{-2\rho}u q^{2\rho} }^{(q)}.
$$
By 4.9(1) of \cite{J}, $S^2(u)=q^{-2\rho}u q^{2\rho} $. Thus, if $P_t^{(q)} = \left(\mathrm{id} \otimes \state{\cdot}_t^{(q)}\right)\circ \Delta$, then Proposition 1.2(1) of \cite{D} implies that $P_t^{(q)}(u)$ is central if $u$ is central. Note that when $q=1$, then the quantum trace reduces to the usual trace.
\end{remark}

\section{Connections to random surface growth}\label{Ctcp}
In this section, we will show the relationship between the non--commutative random walk and a (2+1)--dimensional random surface growth model. First, here is a description of the model, which was introduced in \cite{BF}.
\subsection{Random surface growth}
Consider the two--dimensional lattice $\Z\times\Z_+$. On each horizontal level $\Z\times\{n\}$ there are
exactly $n$ particles, with at most one particle at each lattice site. Let $\tilde{X}^{(n)}_1>\ldots>\tilde{X}^{(n)}_n$ 
denote the $x$--coordinates of the locations of the $n$ particles. Additionally, the particles need to 
satisfy the \textit{interlacing property} $\tilde{X}^{(n+1)}_{i+1} < \tilde{X}^{(n)}_i \leq \tilde{X}^{(n+1)}_i.$  
The particles can be viewed as a random stepped surface, see Figure \ref{Figure}. This can be made rigorous by defining the height function at $(x,n)$ to be the number of particles to the right of $(x,n)$.

\begin{figure}
\captionsetup{width=0.8\textwidth}
\centering

\caption{ The particles as a stepped surface.  The lattice is shifted to make the visualization easier.}

\includegraphics[height=5cm]{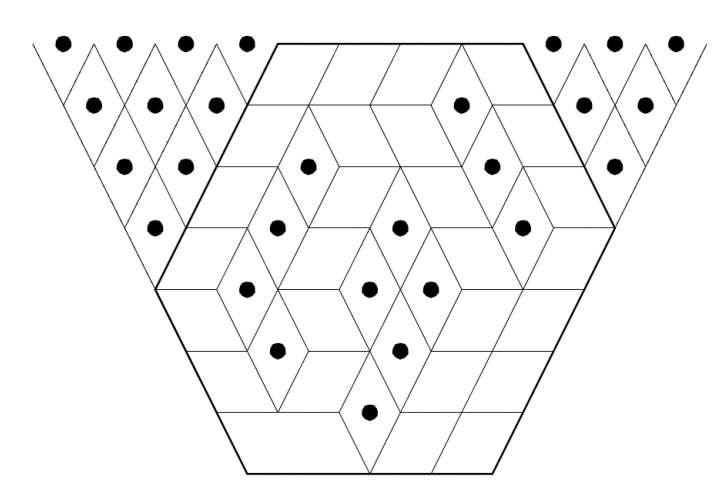}

\label{Figure}
\end{figure}

The dynamics on the particles are as follows. The initial condition is the \textit{densely packed} initial
condition, $\tilde{X}_i^{(n)}=-i+1,1\leq i\leq n$. Each particle has a clock with exponential waiting time of 
rate $1$, with all clocks independent of each other. When the clock rings, the particle attempts to jump
one step to the right. However, it must maintain the interlacing property. This is done by having  
particles push particles above it, and jumps are blocked by particles below it. One can think of lower
particles as being more massive. See Figure \ref{Jumps} for an example.

It turns out to be more convenient to use the co--ordinates $X_i^{(n)}=\tilde{X}_i^{(n)}+i-1$. Then on each level, $X_1^{(n)} \geq \ldots \geq X_n^{(n)}$ and the interlacing property becomes ${X}^{(n+1)}_{i+1} \leq {X}^{(n)}_i \leq {X}^{(n+1)}_i.$ The initial condition is $X_i^{(n)}(0)=0$.

Review some information about these probability measures and dynamics. By a result from \cite{BF,BK},
$$
e^{t\mathrm{Tr} (U-\mathrm{Id})} = \sum_{\mu} \mathrm{Prob}\left(X^{(N)}(t) = \mu \right)\frac{\chi_{\mu}(U)}{\dim\mu}
$$
where $\chi_{\mu}$ and $\dim\mu$ are the character and dimension of the highest weight representation $\mu$. By Theorem 3.1 of \cite{K1} (with $\boldsymbol{\theta}=(1,\ldots,1)$ in the statement of that theorem), for $t\geq s\geq 0$,
$$
\mathbb{P}\left(X^{(N)}(t) = \tau \vert X^{(N)}(s) = \lambda\right) = \sum_{\mu} \mathbb{P}\left( X^{(N)}(t-s)=\mu\right) c_{\lambda\mu}^{\tau} \frac{\dim\tau}{\dim\lambda \dim\mu}.
$$
where $c_{\lambda\mu}^{\tau}$ are the Littlewood--Richardson coefficients defined by
$$
\chi_{\lambda} \cdot \chi_{\mu} = \sum_{\tau} c_{\lambda\mu}^{\tau} \chi_{\tau}
$$
And therefore
\begin{equation}\label{LRC}
\begin{aligned}
e^{t\mathrm{Tr} (U-\mathrm{Id})}  \frac{\chi_{\lambda}}{\dim\lambda} &= \sum_{\mu} \mathrm{Prob}\left(X^{(N)}(t) = \mu\right)\frac{\chi_{\mu}\cdot \chi_{\lambda}}{\dim\mu \dim\lambda}\\
&=\sum_{\mu,\tau} \mathrm{Prob}\left(X^{(N)}(t) = \mu\right)  \frac{\chi_{\tau}}{\dim\tau} c_{\lambda\mu}^{\tau} \frac{\dim\tau }{\dim\mu \dim\lambda}\\
&=\sum_{\tau} \frac{\chi_{\tau}}{\dim\tau} \mathbb{P}\left(X^{(N)}(s+t) = \tau \vert X^{(N)}(s) = \lambda\right) 
\end{aligned}
\end{equation}
Furthermore, for all $t\geq 0$,
\begin{equation}\label{Gibbs2}
\mathbb{P}\left(X^{(M)}(t) = \lambda^{(M)} \vert X^{(N)}(t) = \lambda^{(N)}\right) = \Lambda\left( \lambda^{(N)},\lambda^{(M)} \right) , \quad \forall M \leq N
\end{equation}
where recall that $\Lambda$ was defined in \eqref{Gibbs}.

\begin{figure}
\captionsetup{width=0.8\textwidth}
\caption{The red particle makes a jump. If any of the black particles attempt to jump, their jump is 
blocked by the particle below and to the right, and nothing happens. White particles are not blocked.}
\centering
\includegraphics[height=6cm]{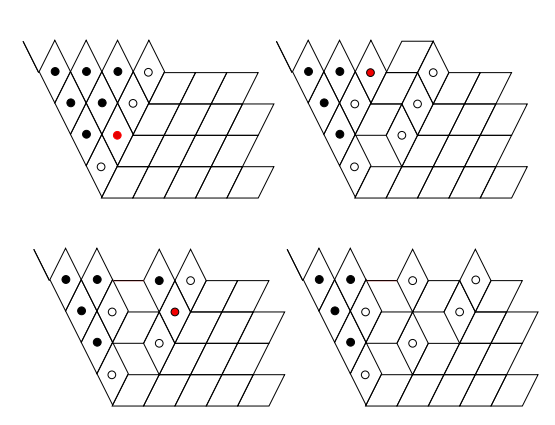}
\label{Jumps}
\end{figure}

\subsection{Restriction to center}

Let $Q_t$ be the Markov operator of the particle system on the $n$th level, which 
defines an operator $Q_t$ on $\mathrm{Fun}(\mathbb{GT}_n)$ by

$$
Q_t f(\lambda) := \sum_{\mu \in \mathbb{GT}_n} Q_t(\lambda \rightarrow \mu)f(\mu).
$$

Given $u\in \Uq$, there is a corresponding observable $O_u$ on $\mathbb{GT}_n$ given by 
\begin{equation}\label{DefinitionOfO}
O_u(\lambda) := \frac{1}{\dim\lambda} \mathrm{Tr}|_{V_{\lambda}}(u)  ,
\end{equation}
where $\chi_{\lambda}$ was defined in \eqref{Chi}. Observe that the map $O:\Uq \rightarrow \mathrm{Fun}(\mathbb{GT}_n)$ is a linear map. Let $\mathcal{F}\subset \mathrm{Fun}(\mathbb{GT}_n)$ be the image of $O$. 

The definition of $O_u$ can be extended further. If $m\leq n$ and $u\in \mathcal{U}\left(\mathfrak{gl}_m \right),$ then define $O_u(\lambda)$ on $\lambda \in \mathbb{GT}_n$ by
\begin{equation}\label{DefinitionofOO}
O_u(\lambda) = \sum_{\mu \in \mathbb{GT}_m} \Lambda(\lambda,\mu) O_u(\mu).
\end{equation}
By the definition of $\Lambda$, \eqref{DefinitionOfO} still holds.

\begin{proposition}\label{OneLevel} (1) For all $u\in \Uq$ and $t\geq 0$,
$$
\langle u\rangle_t = \mathbb{E}[O_u(\lambda(t))]
$$

(2) If $f=O_u\in \mathcal{F}$, then $Q_t f = O_{P_t u}$. In particular, $Q_t$ preserves the image of $O$. 
\end{proposition}
\begin{proof}

(1) By definition
\begin{align*}
\langle u \rangle_t &=\sum_{\mu \in \mathbb{GT}_n} \mathbb{P}(\lambda(t) =\mu) \frac{ \mathrm{Tr}|_{V_{\mu}} u}{\dim\mu}\\
&= \sum_{\mu\in\mathbb{GT}_n} \mathbb{P}(\lambda(t) =\mu) O_u(\mu)\\
&= \mathbb{E}[O_u(\lambda(t))]
\end{align*}

(2) By linearity and \eqref{LRC}
$$
\state{u}_{\chi^t  \tfrac{ \chi_{\lambda}}{\dim\lambda}} =\sum_{\tau} \mathbb{P}\left(X^{(N)}(t) = \tau \vert X^{(N)}(s) = \lambda\right) \state{u}_{\tfrac{\chi_{\tau}}{\dim\tau} }
$$
or equivalently (by Theorem \ref{QRW}(2))
$$
O_{P_t u}(\lambda) = \sum_{\tau} \mathbb{P}\left(X^{(N)}(t) = \tau \vert X^{(N)}(s) = \lambda\right)  O_u(\tau) = \left(Q_tO_u\right)(\lambda).
$$

\end{proof}

The next theorem shows the multi--level relationship between the QRWQG and the random surface growth. This is similar to Theorem 4.5 of \cite{K}. However, the proof there is no longer valid because the center is not preserved unless $q=1$. The extra ingredient here is \eqref{LRC}, which had not been used previously.

\begin{theorem}
Assume that $q_j$ are nonzero real numbers. Suppose that $N_1 \geq \cdots \geq N_r, t_1 \leq \cdots \leq t_r$. Let  $Z_j \in Z(\mathcal{U}_{q_j}(gl_{N_j}))$ for $1\leq j \leq r-1$ and $Y_r \in \mathcal{U}_{q}(gl_{N_r})$. Set
$$
X^{(j)} = \underbrace{1 \otimes \cdots \otimes 1}_{j-1} \otimes  Z_j  \otimes \underbrace{1 \otimes \cdots \otimes 1}_{r-j}.
$$
for $1\leq j \leq r-1$ and $X^{(r)} = 1^{\otimes r-1} \otimes Y_r$.
Then  
\begin{multline*}
\omega\left(j_{1}^{(r)}(X_1^{(r)})\cdots j_{{r-1}}^{(r)}(X_{r-1}^{(r)}) j_{r}^{(r)}(X^{(r)})\right)\\
= \mathbb{E}\left[ O_{X_1}\left(\lambda^{(N_1)}(t_1)\right) \cdots O_{X_{r-1}}\left(\lambda^{(N_{r-1})}(t_{r-1}) \right)  O_{Y_r}\left(\lambda^{(N_r)}(t_r) \right) \right]
\end{multline*}
\end{theorem}
\begin{proof}
First, note that by Theorem \ref{QRW}(5b),
$$
\omega\left(j_{1}^{(r)}(X_1)\cdots j^{(r)}_{{r-1}}(X_{r-1}) j^{(r)}_{r}(Y_r)\right)
=\langle X_1 P^{(r)}_{t_2-t_1}\left( X_2 P^{(r)}_{t_3-t_2} \left( X_3 \cdots P^{(r)}_{t_r-t_{r-1}} Y_r \right)\right)\rangle_{t_1} 
$$
For the remainder of the proof, proceed by induction on $r$. When $r=1$ the result is Proposition \ref{OneLevel}(2).

Assume the statement for some $r$. Then setting $Y_r = X_r \cdot P^{(r)}_{t_{r+1}-t_r}Y_{r+1} $, the induction hypothesis implies 
\begin{multline*}
\langle X_1 P^{(r)}_{t_2-t_1}\left( X_2 P^{(r)}_{t_3-t_2} \left( X_3 \cdots P^{(r)}_{t_r-t_{r-1}} \left(X_r P^{(r)}_{t_{r+1}-t_r}Y_{r+1} \right)\right)\right)\rangle_{t_1} \\
= \mathbb{E}\left[ O_{X_1}\left(\lambda^{(N_1)}(t_1)\right) \cdots O_{X_{r-1}}\left(\lambda^{(N_{r-1})}(t_{r-1}) \right)  O_{Y_r}\left(\lambda^{(N_r)}(t_r) \right) \right]
\end{multline*}
By the definition of an expectation, this equals (where the summation over each $\mu^{(m)}$ is over $\mathbb{GT}_m$)
\begin{equation}\label{Step}
\sum_{\mu^{(N_1)},\ldots,\mu^{(N_r)}} \mathbb{P}\left( \lambda^{(N_1)}(t_1) = \mu^{(N_1) } , \ldots ,  \lambda^{(N_r)}(t_r)= \mu^{(N_r) } \right) O_{X_1}\left(\mu^{(N_1)}\right) \cdots O_{X_{r-1}}\left(\mu^{(N_{r-1})}\right)  O_{Y_r}\left(\mu^{(N_r)} \right)
\end{equation}
By the definition of $O$ in \eqref{DefinitionOfO}, and the assumption that $X_r$ is central,
$$
O_{Y_r}(\mu^{(N_r)}) = \frac{ \mathrm{Tr}\vert_{V_{\mu^{(N_r)}}} \left( X_r \cdot P_{t_{r+1}-t_r}Y_{r+1} \right) }{\dim \mu^{(N_r)}}= O_{X_r}(\mu^{(N_r)}) \frac{\mathrm{Tr}\vert_{V_{\mu^{(N_r)}}} \left(P_{t_{r+1}-t_r}Y_{r+1} \right) }{\dim \mu^{(N_r)}}.
$$
The reduction from $P_t^{(r)}$ to $P_t$ follows from 
$$
P_t^{(r)}(X^{(r)}) = (1\otimes^{j-1} \otimes P_t \otimes 1^{\otimes r-j})(X^{(r)})
$$
Furthermore, setting $\tilde{\chi}_{\mu} = \chi_{\mu}/\dim\mu$,
\begin{align*}
& \left(\dim \mu^{(N_r)}\right)^{-1}\mathrm{Tr}\vert_{V_{\mu^{(N_r)}}} \left(P_{t_{r+1}-t_r}Y_{r+1} \right) = \langle P_{t_{r+1}-t_r}Y_{r+1}\rangle_{  \tilde{\chi}_{\mu^{(N_r)}}   } = \state{    Y_{r+1} }_{\tilde{\chi}_{\mu^{(N_r)}}  \chi_{t_{r+1}-t_r}}\\
 & \stackrel{\eqref{LRC}}{=} \sum_{\nu^{(N_r)}} \mathbb{P}\left( \lambda^{(N_r)}(t_{r+1}) = \nu^{(N_r)} \vert \lambda^{(N_{r})}(t_r) = \mu^{(N_r)} \right)O_{Y_{r+1}}\left(\nu^{(N_r)}\right)\\
 & \stackrel{\eqref{DefinitionofOO}}{=} \sum_{\nu^{(N_r)}, \mu^{(N_{r+1})}} \mathbb{P}\left( \lambda^{(N_r)}(t_{r+1}) = \nu^{(N_r)} \vert \lambda^{(N_{r})}(t_r) = \mu^{(N_r)} \right)O_{Y_{r+1}}\left(\mu^{(N_{r+1})}\right) \Lambda\left(\nu^{(N_r)},\mu^{(N_{r+1})}\right)
\end{align*}

Therefore, by \eqref{Gibbs2} the expression \eqref{Step} equals
\begin{multline*}
\sum_{\substack{\mu^{(N_1)},\ldots,\mu^{(N_r)} \\ \nu^{(N_r)},\mu^{(N_{r+1})}}} \mathbb{P}\left( \lambda^{(N_1)}(t_1) = \mu^{(N_1) } , \ldots ,  \lambda^{(N_r)}(t_r)= \mu^{(N_r) } , \lambda^{(N_r)}(t_{r+1}) = \nu^{(N_r)}, \lambda^{(N_{r+1})}(t_{r+1}) = \mu^{(N_{r+1})}\right) \\
\times O_{X_1}\left(\mu^{(N_1)}\right) \cdots O_{X_{r-1}}\left(\mu^{(N_{r-1})}\right)  O_{X_r}\left(\mu^{(N_r)} \right)O_{Y_{r+1}}\left(\mu^{(N_{r+1})} \right)
\end{multline*}
Because there is no observable in $\nu^{(N_r)}$, the sum over $\nu^{(N_r)}$ can be eliminated, completing the proof.
\end{proof}

Although there is not a rigorous way to multiply elements of $\Uq$ and $\mathcal{U}_{q_1}(\mathfrak{gl}_n)$, it is not unreasonable to conjecture that the results in this section should still be true if the multiplication is interpreted formally. Here is a (numeric) example of how to do this.

\begin{example}
Consider an irreducible representation of $\mathcal{U}_q(\mathfrak{gl}_2)$ with highest weight $(\lambda_1,\lambda_2)$. The weights can be written as $(\lambda_1-j,\lambda_2+j)$ for $0\leq j \leq \lambda_1-\lambda_2$. One can check that $E_{21}E_{12} \in \mathcal{U}_{q_2}(\mathfrak{gl_n})$ acting on $(\lambda_1-j,\lambda_2+j)$ multiplies by the constant
$$
[\lambda_1-\lambda_2]_{q_2} + [\lambda_1-\lambda_2-2]_{q_2} + \ldots + [\lambda_1-\lambda_2-2j+2]_{q_2}.
$$
So that $q_1^{E_{11}+E_{22}} \otimes E_{21}E_{12}$ with the quantum trace $\mathrm{Tr}_s$ acts as the observable
\begin{multline*}
\mathcal{O}(\lambda_1,\lambda_2) \\= \frac{1}{\lambda_1-\lambda_2+1}\sum_{j=1}^{\lambda_1-\lambda_2} s^{\lambda_1-\lambda_2-2j}q_1^{\lambda_1+\lambda_2} \left([\lambda_1-\lambda_2]_q + [\lambda_1-\lambda_2-2]_q + \ldots + [\lambda_1-\lambda_2-2j+2]_q\right)
\end{multline*}
Now for $q_1=-0.27+{\color{black}0.3}i,q_2=0.8,s=0.6,t=0.31,(a_1,a_2)=(3,1)$, the determinantal formula  from section 2.3 of \cite{BF} predicts {\color{black} (where $f_t(k)$ is the $t^k$ coefficient of $e^{t}$) that $[P_t(\mathcal{O})](a_1,a_2)$ equals}
\begin{multline*}
e^{-2t} \sum_{\lambda_1=a_1}^{\infty} \sum_{\lambda_2=a_2}^{\lambda_1} \mathcal{O} \cdot \frac{ \lambda_1-(\lambda_2-1)}{{\color{black} a_1-(a_2-1)}} \det
\left(
\begin{array}{cc}
f_t(\lambda_1-a_1) &f_t(\lambda_1-(a_2-1)) \\
f_t((\lambda_2-1)-a_1) & f_t((\lambda_2-1)-(a_2-1)) \\
\end{array}
\right) \\
\approx -0.02788676811357415 - 0.002852163596477639i 
\end{multline*}
for the summation over $\lambda_1$ up to $50\approx\infty$. 

Formally, the co--product $\Delta^{\otimes 2}$ applied to $q_1^{E_{11}+E_{22}} \otimes E_{21}E_{12}$ yields 
\begin{multline}\label{CUP}
\Delta^{\otimes 2} \left(q_1^{E_{11}+E_{22}} \otimes E_{21}E_{12}\right) \\
= q_1^{E_{11}+E_{22}} \otimes q_1^{E_{11}+E_{22}} \otimes E_{21}E_{12} \otimes q_2^{E_{22}-E_{11}} + q_1^{E_{11}+E_{22}} \otimes q_1^{E_{11}+E_{22}}  \otimes q_2^{E_{11}-E_{22}} \otimes E_{21}E_{12}
\end{multline}
By applying $\state{\cdot}_t \otimes \state{\cdot}_{\epsilon} \otimes \state{\cdot}_t \otimes \state{\cdot}_{\epsilon}$ to both sides and taking $\epsilon\rightarrow 0$, Theorem \ref{QRW}(2) implies that $\langle q_1^{E_{11}+E_{22}} \otimes E_{21}E_{12} \rangle_t$ with the quantum trace at $s$ solves the differential equation
$$
y'(t) = (q_1q^{-1}s - 1 + q_1qs^{-1} - 1)y(t)  + q_1 s^{-1} \exp\left(t(q_1qs - 1 + q_1q^{-1}s^{-1}-1)\right), \quad y(0)=0
$$
which is solved by 
$$
y(t) = e^{t(q_1qs^{-1}+q_1q^{-1}s-2)}q(e^{q_1q^{-1}s^{-1}t(q^2-1)(s^2-1)}-1 ) (q^2-1)^{-1}(s^2-1)^{-1} .
$$
Furthermore, applying $\left( \mathrm{id} \otimes \langle \cdot \rangle_t^{(s)} \right)$ to \eqref{CUP}
\begin{align*}
P_t^{(s)}\left(q_1^{E_{11}+E_{22}} E_{21}E_{12}\right) &= \langle q_1^{E_{11}+E_{22}}q^{E_{22}-E_{11}} \rangle_t^{(s)} q_1^{E_{11}+E_{22}} E_{21}E_{12} + \langle q_1^{E_{11}+E_{22}}E_{21}E_{12} \rangle_t^{(s)}  q_1^{E_{11}+E_{22}}q^{E_{11}-E_{22}} \\
&=   e^{t(q_1q^{-1}s+q_1qs^{-1}-2)} q_1^{E_{11}+E_{22}} E_{21}E_{12} + y(t) q_1^{E_{11}+E_{22}}q^{E_{11}-E_{22}}
\end{align*}
which predicts {\color{black} that $[P_t(\mathcal{O})](a_1,a_2)]$ equals}
\begin{multline*}
e^{t(q_1q^{-1}s+q_1qs^{-1}-2)}  \cdot \mathcal{O} + \frac{{\color{black} y(t)}}{ {\color{black} a_1-(a_2-1)}} \sum_{j=0}^{a_1-a_2} s^{a_1-a_2-2j} q_1^{a_1+a_2}q^{a_1-j-(a_2+j)}\\
= -0.02788676811357414 - 0.002852163596477645 i
\end{multline*}
which matches to 17 decimal points.

\end{example}

\section{Asymptotic Gaussian Fluctuations}\label{Asympt}

By (46) in \cite{GZB}, the element 
$$
C^{(n)} = C^{(n)}_q :=  \sum_{i=1}^n q^{2i-2n}q^{2E_{ii}}+ (q-q^{-1})^2 \sum_{1 \leq i < j \leq n} q^{2j-2n-1}q^{E_{ii}+E_{jj}}E_{ij}E_{ji}
$$
is central in $\Uq$. When acting on the lowest weight vector of $V_{\lambda}$, the second term vanishes, so $C^{(n)}_q$ acts as the constant (see also (51) in \cite{GZB})
\begin{equation}\label{Series}
\sum_{i=1}^n q^{2(\lambda_i-i+n)} = \sum_{i=1}^n \sum_{k=0}^{\infty} \frac{(2h)^k(\lambda_i-i+n)^k}{k!} =: \sum_{k=0}^{\infty} \frac{(2h)^k}{k!} \Psi_k^{(n)}(\lambda)
\end{equation}
where $q=\exp h.$ By previously known results (\cite{BB,BF,K}), there are fixed--time asymptotics: if $N_j=[\eta_j L],t = [\tau L]$ then $\mathbb{E}\Psi_k^{(N)}(\lambda(t)) \sim L^{k+1}$ (where $\lambda(t)$ is distributed as $ X^{(N)}(t)$) and 
\begin{equation}\label{Previously}
\left( \frac{\Psi_{k_1}^{(N_1)}(\lambda(t)) - \mathbb{E}\Psi_{k_1}^{(N_1)}(\lambda(t)) }{L^{k_1}}  , \ldots, \frac{\Psi_{k_r}^{(N_r)}(\lambda(t)) - \mathbb{E}\Psi_{k_r}^{(N_r)}(\lambda(t)) }{L^{k_r}} \right) \rightarrow (\xi_1,\ldots,\xi_r)
\end{equation}
where $(\xi_1,\ldots,\xi_r)$ is a Gaussian vector with mean zero covariance
$$
\mathbb{E}[\xi_i\xi_j]=\mathbb{E}\left[\mathfrak{G}(k_i,\eta_i,\tau)\mathfrak{G}(k_j,\eta_j,\tau)\right].
$$
By \eqref{Series}, this suggests that $q_j$ should depend on $L$ as $q_j = \exp h_j/L$. This scaling also suggests that $\state{ C_q^{(N)}}_t$ should be of order $\sim L$ with fluctuations of constant order, which will be confirmed below.

In order to apply the fixed--time asymptotics in \eqref{Previously} to the observables in \eqref{Series}, we need to justify the interchange of the limit and the infinite summation. This is done in the proof of the next proposition. 

\begin{proposition}\label{FixedTime}
Suppose that $N_j = [\eta_j L], t = \tau L$ and $q_j = \exp(h_j/L)$ for $1\leq j \leq r$. Then as $L\rightarrow\infty$,
$$
\left( C_{q_1}^{(N_1)} - \state{C_{q_1}^{(N_1)} }_{t}  , \cdots, C_{q_r}^{(N_r)} - \state{C_{q_r}^{(N_r)} }_{t}   \right) \rightarrow (\mathcal{G}(h_1,\eta_1,\tau),\ldots, \mathcal{G}(h_r,\eta_r,\tau))
$$
with respect to the state $\omega(\cdot)$.
\end{proposition}
\begin{proof}
The proposition amounts to proving the following statement:
$$
\Big\langle \left(C_{q_1}^{(N_1)} - \state{C_{q_1}^{(N_1)} }_{t} \right)  \cdots \left(C_{q_s}^{(N_s)} - \state{C_{q_s}^{(N_s)} }_{t} \right) \Big\rangle_t \longrightarrow \mathbb{E}\left[  \mathcal{G}(h_1,\eta_1,\tau) \cdots \mathcal{G}(h_s,\eta_s,\tau)  \right]
$$
By {\color{black} Proposition \ref{OneLevel}(1)} and the series \eqref{Series}, it is equivalent to prove
\begin{multline*}
\sum_{k_1=0}^{\infty} \cdots \sum_{k_s=0}^{\infty} \frac{(2h_1)^{k_1}}{k_1!} \cdots \frac{(2h_s)^{k_s}}{k_s!} \Bigg\langle \left(  \frac{\Psi_{k_1}^{(N_1)}(\lambda(t)) - \mathbb{E}\Psi_{k_1}^{(N_1)}(\lambda(t)) }{L^{k_1}}  \right) \cdots \left(  \frac{\Psi_{k_s}^{(N_s)}(\lambda(t)) - \mathbb{E}\Psi_{k_s}^{(N_s)}(\lambda(t)) }{L^{k_s}}  \right)  \Bigg\rangle_t  \\
\rightarrow \mathbb{E}\left[  \mathcal{G}(h_1,\eta_1,\tau) \cdots \mathcal{G}(h_s,\eta_s,\tau)  \right].
\end{multline*}
If the order of the summation and the limit can be changed, then the result follows from \eqref{Previously}. In order to change the order, we will use the dominated convergence theorem and show the bound
\begin{multline*}
 \left| \Bigg\langle \left(  \frac{\Psi_{k_1}^{(N_1)}(\lambda(t)) - \mathbb{E}\Psi_{k_1}^{(N_1)}(\lambda(t)) }{L^{k_1}}  \right) \cdots \left(  \frac{\Psi_{k_s}^{(N_s)}(\lambda(t)) - \mathbb{E}\Psi_{k_s}^{(N_s)}(\lambda(t)) }{L^{k_s}}  \right)  \Bigg\rangle_t  \right| \\
 \leq \mathbb{E}\left[\mathfrak{G}(k_1,\eta_1,\tau) \cdots \mathfrak{G}(k_s,\eta_s,\tau)\right] + A(\tau, \eta_1,\ldots,\eta_s)^{k_1+\cdots k_s} B_{k_1}B_{k_2}\cdots B_{k_s} ,
\end{multline*}
where $A(\tau, \eta_1,\ldots,\eta_s)$ is a quantity depending on the values of $\tau, \eta_1,\ldots,\eta_s$, but not on the values of $k_1,\ldots,k_s$, and $B_k$ is the $k$--th Bell number. To see that this bound is sufficient, we use a probabilistic argument. Let $X$ be a Poisson random variable of rate $A$. Then its characteristic function is given by 
$$
\mathbb{E}[e^{i t X}] = e^{A(e^{it}-1)}, 
$$
so in particular, $\mathbb{E}[e^X] = e^{A(e-1)}<\infty$. Since $X$ is non--negative, this means that Fubini's theorem can be applied to conclude that
$$
\mathbb{E}[e^X] = \mathbb{E}\left[ \sum_{k=0}^{\infty} \frac{X^k}{k!} \right] = \sum_{k=0}^{\infty} \mathbb{E}\left[ \frac{X^k}{k!} \right] = \sum_{k=0}^{\infty}  \frac{T_k(A)}{k!} < \infty.
$$
where $T_k(x)$ is the Touchard polynomial.\footnote{Recall the following bounds, which are valid for all positive integers $k$:
\begin{align*}
\sqrt{2 \pi} k^{k+1/2} e^{-k} &< k! \\
B_k &< \left( \frac{0.792k}{\ln(k+1)}\right)^k.
\end{align*}
These two bounds imply the inequality
$$
\sum_{k=0}^{\infty} \frac{B_k}{k!}A^k < \sum_{k=1}^{\infty} \frac{1}{\sqrt{2\pi} k^{1/2}} \left( \frac{Ae\cdot 0.792}{\ln(k+1)} \right)^k < \infty.
$$
}

To show the bound, we will use explicit generators of the center $Z(U_q(\mathfrak{gl}_m))$.  These were first discovered in section 7 of  \cite{GKLLRT}; see also chapter 7 of \cite{kn:M} for an exposition. 

Let $\mathcal{G}_m$ denote the directed graph with vertices and edges
$$
\{1,\ldots,m\} \quad \{(i,j):1\leq i,j\leq m\}. 
$$
Let $\Pi_k^{(m)}$ denote the set of all paths in $\mathcal{G}_m$ of length $k$ which start and 
end at the vertex $m$. For $\pi\in \Pi_k^{(m)}$ let $r(\pi)$ denote the length of the first
return to $m$. Let $E(\pi)\in U(\mathfrak{gl}_m)$ denote the element with coefficient $r(\pi)$ obtained by taking
the product when labeling
the edge $(i,j)$ with $E_{ij}$ when $i\neq j$, and the edge $(i,i)$ with $E_{ii}-m+1$.
For example, the path 
$$
\pi=\{5\rightarrow 3\rightarrow 3\rightarrow 1\rightarrow 5\rightarrow 5\rightarrow 2\rightarrow 5\}
$$
is in $\Pi^{(5)}_7$ with $r(\pi)=4$ and 
$$
E(\pi)=4E_{53}(E_{33}-4)E_{31}E_{15}(E_{55}-4)E_{52}E_{25}.
$$
Define the elements
$$
\Psi^{(N)}_k := \sum_{m=1}^N \sum_{\pi\in \Pi_k^{(m)}} E(\pi) \in U(\mathfrak{gl}_N).
$$
For example,
$$
\Psi^{(N)}_1 = \sum_{m=1}^N (E_{mm}-m+1), \quad \Psi^{(N)}_2 = \sum_{m=1}^N (E_{mm}-m+1)^2 + 2\sum_{1\leq l<m\leq N} E_{ml}E_{lm}.
$$
When acting on the highest weight vector of $V_{\lambda}$, only the diagonal terms make a contribution, so it quickly follows that $\Psi_k^{(N)}$ acts as the constant $\Psi_k^{(N)}(\lambda)$.

The bound 
$$
\vert \Pi_k^{(m)} \vert < m^k
$$
implies the bound
$$
\sum_{m=1}^N \vert \Pi_k^{(m)} \vert \leq N^{k+1},
$$
and can be absorbed into the value $A(\tau,\eta_1,\ldots,\eta_s)$. The constant $r(\pi)$ can likewise be absorbed. An upper bound for $\langle E(\pi) \rangle_t$ for $E(\pi) \in \Pi_k^{(m)}$ follows from Proposition 4.1 of \cite{K}; it is given by the number of partitions of the set $\{1,\ldots,k\}$, which is the $k$th Bell number $B_k$. 

This shows that the summand is bounded by a summable function, so the limit can be changed with the summation, showing the result.
\end{proof}

For multi--time asymptotics, it is also necessary to find the states of each monomial in $C^{(n)}_q$. Below, recall that 
$$
\Hyp(-m;2;-x)= \sum_{r=1}^{m+1} \binom{m}{r-1} \frac{x^{r-1}}{r!}.
$$

\begin{proposition} \label{fij}
Assume that $q$ is not a root of unity. For $i<j$, 
$$
\langle q^{E_{ii}+E_{jj}} E_{ij}E_{ji} \rangle_{\gamma} = q\gamma e^{\gamma(q^2-1)} {}_{1}F_1(-(j-i-1);2;-(q-q^{-1})^2\gamma) =: f_{j-i}(\gamma).
$$
\end{proposition}
\begin{proof}
By \eqref{CoProd}, for $1 \leq i < j \leq n$
\begin{align*}
 \Delta (q^{E_{ii}+E_{jj}}  E_{ij} E_{ji}  ) 
& =   q^{E_{ii}+E_{jj}}E_{ij} E_{ji} \otimes q^{2E_{jj}} + q^{2E_{ii}} \otimes q^{E_{ii}+E_{jj}}E_{ij}E_{ji}  \\
& + (q-q^{-1})^2 \sum_{r=i+1}^{j-1} \sum_{k=i+1}^{j-1} q^{E_{ii}+E_{rr}}E_{ir}E_{ki} \otimes q^{E_{ii}+E_{jj}} E_{rj} q^{E_{kk}-E_{ii}} E_{jk} 
\end{align*}
Hence, by Theorem \ref{QRW}(2), 
\begin{equation} \label{semigroup1}
\begin{aligned}
 \state{ q^{E_{ii}+E_{jj}}  E_{ij}E_{ji}}_{\gamma+\eps} & = \state{ q^{E_{ii}+E_{jj}} E_{ij}E_{ji}}_{\gamma} \state{q^{2E_{ii}}}_{\eps} + \state{q^{2E_{jj}}}_{\gamma} \state{ q^{E_{ii}+E_{jj}} E_{ij}E_{ji}}_{\eps}  \\
 & + (q-q^{-1})^2q^{-1}\sum_{r=i+1}^{j-1} \state{q^{E_{rr}+E_{jj}}E_{rj}E_{jr}}_{\gamma}\state{q^{E_{ii}+E_{rr}}E_{ir}E_{ri}}_{\eps}
\end{aligned}
\end{equation} 
where we have used $q^{E_{rr}-E_{ii}} E_{rj}= E_{rj} q^{E_{rr}-E_{ii}}q^{(\epsilon_r,\epsilon_r-\epsilon_{r+1})} =E_{rj} q^{E_{rr}-E_{ii}}q $ for $r<j$. In particular,
\begin{multline*}
\state{ q^{E_{ii}+E_{jj}} E_{ij}E_{ji}}_{\gamma+\eps} = (1+\eps(q^2-1))\state{ q^{E_{ii}+E_{jj}} E_{ij}E_{ji}}_{\gamma} \\
+ q{\eps} \state{q^{2E_{jj}}}_{\gamma} + (q-q^{-1})^2q^{-1}\sum_{r=i+1}^{j-1} \state{q^{E_{rr}+E_{jj}}E_{rj}E_{jr}}_{\gamma}q{\eps} + O(\epsilon^2). 
\end{multline*}
Therefore, $f_{j-i}(\gamma) := \state{ q^{E_{ii}+E_{jj}} E_{ij}E_{ji}}_{\gamma}$ satisfies the differential equation
$$
f_{j-i}'(\gamma) = (q^2-1)f_{j-i}(\gamma)  + q e^{\gamma(q^2-1)} + (q-q^{-1})^2 \sum_{r=i+1}^{j-1}f_{j-r}(\gamma).
$$
In general, if a family of functions $\{f_m(t)\}$ satisfies the differential equation
$$
f'_m(x) = a f_m(x) +  g(x) + b\sum_{i=1}^{m-1}   f_{i}(x) 
$$
then using integrating factors shows that $f_m$ is solved by
\begin{align*}
f_m(x) &= e^{ax} \int_0^x e^{-ax_1} \left( g(x_1) + b\sum_{i_1=1}^{m-1} f_{i_1}(x_1)\right)dx_1 \\
&= e^{ax} \int_0^x e^{-ax_1} g(x_1)dx_1 + e^{ax} b \sum_{i_1=1}^{m-1} \int_0^x \int_0^{x_1} e^{-ax_2}\left( g(x_2) + \sum_{i_2=1}^{i_1-1} f_{i_2}(x_2)\right)dx_2 dx_1 \\
&= e^{ax} \int_0^x e^{-ax_1}g(x_1)dx_1 + e^{ax} b \binom{m-1}{1} \int_0^x \int_0^{x_1} e^{-ax_2}g(x_2)dx_2dx_1 + \ldots \\
&= \sum_{r=1}^{m} e^{ax} \binom{m-1}{r-1} b^{r-1} \int_0^x \cdots \int_0^{x_{r-1}} e^{-ax_r}g(x_r)dx_r \cdots dx_1
\end{align*}
Applying this with $a=(q^2-1), b=(q-q^{-1})^2$ and $g(t)=qe^{t(q^2-1)}$ yields
\begin{align*}
f_{ij}(\gamma) &= q e^{\gamma(q^2-1)} \sum_{r=1}^{j-i} \binom{j-i-1}{r-1}(q-q^{-1})^{2(r-1)}\frac{\gamma^r}{r!}\\
&=q\gamma e^{\gamma(q^2-1)} {}_1F_1(-(j-i-1);2;-(q-q^{-1})^2\gamma)
\end{align*}
\end{proof}

The following result is generalization for $n=2$ in \cite{Bi3} (see also chapter 13 in \cite{B2}).

\begin{proposition}\label{Coeff}
Assume that $q$ is not a root of unity. For any $t\geq 0$,
\begin{equation}
P_t(C^{(n)})  = e^{t(q^2-1)} \cdot C^{(n)} + \sum_{k=1}^{n-1} A_{k}(t) C^{(n-k)} 
\end{equation}
where
$$
A_{k}(t)  =  q^{-1}(q-q^{-1})^2 f_{k}(t).
$$
\end{proposition}
\begin{proof}
Since $\Delta(q^{2E_{ii}}) = q^{2E_{ii}} \otimes q^{2E_{ii}}$, 
$$
P_t(q^{2E_{ii}}) = e^{t(q^2-1)} q^{2E_{ii}}
$$
Now using \eqref{CoProd}, we have
$$
\begin{aligned}
P_t(q^{E_{ii}+E_{jj}}  E_{ij} E_{ji}) &= e^{t(q^2-1)}q^{E_{ii}+E_{jj}}E_{ij} E_{ji} + \langle q^{E_{ii}+E_{jj}}E_{ij}E_{ji} \rangle_t q^{2E_{ii}} \\
& + q^{-1}(q-q^{-1})^2 \sum_{r=i+1}^{j-1} \langle q^{E_{rr}+E_{jj}} E_{rj} E_{jr} \rangle_{t} q^{E_{ii}+E_{rr}}E_{ir}E_{ri}  
\end{aligned}
$$
where we have used $q^{E_{rr}-E_{ii}} E_{rj}= E_{rj} q^{E_{rr}-E_{ii}}q^{(\epsilon_r,\epsilon_r-\epsilon_{r+1})} =E_{rj} q^{E_{rr}-E_{ii}}q $ for $r<j$. Because the term $e^{t(q^2-1)}$ occurs as a coefficient in both $q^{2E_{ii}}$ and $q^{E_{ii}+E_{jj}}E_{ij}E_{ji}$, we can write
\begin{align*}
P_t(C^{(n)})  & - e^{t(q^2-1)} \cdot C^{(n)} \\
&= (q-q^{-1})^2\sum_{1\leq i<j \leq n} q^{2j-2n-1} \left( f_{j-i}(t)q^{2E_{ii}} + q^{-1}(q-q^{-1})^2 \sum_{r=i+1}^{j-1} f_{j-r}(t)q^{E_{ii}+E_{rr}}E_{ir}E_{ri}   \right)\\
& = (q-q^{-1})^2 \sum_{i=1}^{n-1} \left ( \sum_{j=i+1}^{n} q^{2j-2n-1} f_{j-i}(t) \right ) q^{2E_{ii}}  \\
& \quad + (q-q^{-1})^4 \sum_{1 \leq i < r \leq n-1} \left( \sum_{j=r+1}^{n} q^{2j-2n-2} f_{j-r}(t) \right)  q^{E_{ii}+E_{rr}}E_{ir}E_{ri} .
\end{align*}

Re--arrange the summation to note that
\begin{align*}
&\sum_{k=1}^{n-1} A_k(t) C^{(n-k)} \\
&= \sum_{k=1}^{n-1}\sum_{i=1}^{n-k} A_k(t)q^{2i-2n+2k} q^{2E_{ii}} + (q-q^{-1})^2\sum_{k=1}^{n-1} \sum_{1\leq i<j\leq n-k} A_k(t) q^{2j-2n+2k-1}q^{E_{ii}+E_{jj}}E_{ij}E_{ji} \\
&=\sum_{i=1}^{n-1} \sum_{k=1}^{n-i} A_k(t) q^{2i-2n+2k}q^{2E_{ii}} + (q-q^{-1})^2 \sum_{1 \leq i< j \leq n-1} \sum_{k=1}^{n-j} A_k(t) q^{2j-2n+2k-1}q^{E_{ii}+E_{jj}}E_{ij}E_{ji} \\
&=\sum_{i=1}^{n-1} \sum_{k=1}^{n-i} A_k(t) q^{2i-2n+2k}q^{2E_{ii}} + (q-q^{-1})^2 \sum_{1 \leq i< r \leq n-1} \sum_{k=1}^{n-r} A_{k}(t) q^{2r-2n+2k-1}q^{E_{ii}+E_{rr}}E_{ir}E_{ri} 
\end{align*}
And now setting $k=j-i$ in the first sum and $k=j-r$ in the second sum shows the result. 
\end{proof}

Notice that in the scalings at the beginning of this section,  $\state{q^{E_{ii}+E_{jj}} E_{ij}E_{ji}}_{t}$ is of order $L$ and $\state{q^{2E_{ii}}}_{t}$ is of constant order. This implies that $\state{C_q^{(N)}}_t$ is of order $L$, as expected.

We can now state the convergence. Recall the definition of convergence used in section \ref{Ncp}.

\begin{theorem}\label{Asymp}
Suppose that $N_j = [\eta_j L], t_j = \tau_jL$ and $q_j = \exp(h_j/L)$ for $1\leq j \leq r$. For $1 \leq j \leq r$, let $X^{(j)}$ denote the element
$$
\underbrace{1 \otimes \cdots \otimes 1}_{j-1} \otimes \left( (C_{q_j})^{(N_j)} - \state{ (C_{q_j})^{(N_j)} }_t \right) \otimes \underbrace{1 \otimes \cdots \otimes 1}_{r-j}.
$$
Then as $L\rightarrow\infty$,
$$
\left( j_{1} \left( X^{(1)}\right) , \cdots, j_{r}\left(  X^{(r)}  \right)\right) \rightarrow (\mathcal{G}(h_1,\eta_1,\tau_1),\ldots, \mathcal{G}(h_r,\eta_r,\tau_r))
$$
with respect to the state $\omega(\cdot)$.
\end{theorem}
\begin{proof}
Note that
$$
P_t^{(r)}(X^{(j)}) = \underbrace{1 \otimes \cdots \otimes 1}_{j-1} \otimes P_t\left( (C_{q_j})^{(N_j)} - \state{ (C_{q_j})^{(N_j)} }_t \right) \otimes \underbrace{1 \otimes \cdots \otimes 1}_{r-j}.
$$
Applying Theorem \ref{QRW}(5b) and using that each $P_t C^{(n)}$ is a linear combination of $C^{(k)}$ for $1\leq k\leq n$, this shows that the multi--time fluctuations can be written as a linear combination of fixed--time fluctuations of central elements. Each central element has a series of the form \eqref{Series}, so it follows from \eqref{Previously} that the convergence will be to some Gaussian vector. It remains to show that the covariance is that of $\mathcal{G}$. 

The theorem for fixed--time follows from Proposition \ref{Same} and the discussion at the beginning of this section. By {\color{black}Theorem \ref{QRW}(4b)}, it suffices to calculate the limit of 
$$
\state{ \left(C_{q_1}^{(N_1)} - \state{C_{q_1}^{(N_1)} }_{t_1}\right) \cdot \left(P_{t_2-t_1}C_{q_2}^{(N_2)} - \state{P_{t_2-t_1}C_{q_2}^{(N_2)} }_{t_1} \right)  }_{t_1}.
$$
By Proposition \ref{Coeff}, this equals
\begin{multline*}
e^{(t_2-t_1)(q_2^2-1)}\state{ \left(C_{q_1}^{(N_1)} - \state{C_{q_1}^{(N_1)} }_{t_1}\right) \cdot \left(C_{q_2}^{(N_2)} - \state{C_{q_2}^{(N_2)} }_{t_1} \right)  }_{t_1}. \\
+ \sum_{k=1}^{N_2-1} A_k(t_2-t_1)\state{ \left(C_{q_1}^{(N_1)} - \state{C_{q_1}^{(N_1)} }_{t_1}\right) \cdot \left(C_{q_2}^{(N_2-k)} - \state{C_{q_2}^{(N_2-k)} }_{t_1} \right)  }_{t_1}.
\end{multline*}
In the formula for $A_k(t), $ given by
\begin{align*}
 A_k(t)  &= q^{-1}(q-q^{-1})^2 \cdot f_{k}(t) \\
 &= t (q-q^{-1})^2 e^{t(q^2-1)} \sum_{r=1}^k \binom{k-1}{r-1} \frac{((q-q^{-1})^2t)^{r-1}}{r!}, 
\end{align*}
let $k,t,q$ depend on $L$ as $k = [\kappa L], t= \tau L, q= e^{h/L}$. Then the constant in front of the summation satisfies the limit
$$
\lim_{L\rightarrow \infty} L \cdot t (q-q^{-1})^2 e^{t(q^2-1)}  = \tau(2h)^2 e^{2h\tau}
$$
and each term in the summation satisfies the limit
\begin{align*}
\lim_{L\rightarrow \infty}  \binom{k-1}{r-1} \frac{((q-q^{-1})^2t)^{r-1}}{r!} &= \lim_{L \rightarrow \infty} \frac{L^{-1}(k-1)L^{-1}(k-2) \cdots L^{-1}(k-r+1)}{(r-1)!}   \frac{(L(q-q^{-1})^2t)^{r-1}}{r!} \\
&=  \frac{\kappa^{r-1}}{(r-1)!} \frac{ ( (2h)^2\tau )^{r-1}  }{r!}
\end{align*}
Therefore, by the dominated convergence theorem, 
\begin{align*}
\lim_{L \rightarrow \infty} L \cdot A_k(t) & = e^{2h\tau} \sum_{r=1}^{\infty} \frac{\kappa^{r-1} }{(r-1)!} (2h)^{2r} \frac{\tau^r}{r!} \\
 & =  e^{2h\tau}  \cdot \frac{2h\sqrt{\tau}}{\sqrt{\kappa}} I_1\left( 4h  \sqrt{\kappa\tau} \right)
\end{align*}
where $I_n(x)$ is the modified Bessel function of the first kind:
$$
I_n(x) = \sum_{r=1}^{\infty} \frac{1}{(r-1)! (r-1+n)!}\left( \frac{x}{2}\right)^{2(r-1)+n} = \frac{1}{2\pi i}\oint e^{(x/2)(t+t^{-1})}t^{-n-1}dt,
$$
where the contour encloses the origin in a counterclockwise direction. The sum over $k$ becomes a Riemann sum for an integral over $\kappa$, so therefore the asymptotic limit is
\begin{multline*}
e^{2(\tau_2-\tau_1)h_2} \mathbb{E}[\mathcal{G}(h_1,\eta_1,\tau_1)\mathcal{G}(h_2,\eta_2,\tau_1)] \\
+ e^{2(\tau_2-\tau_1)h_2} \cdot 2h_2\sqrt{\tau_2-\tau_1} \int_0^{\eta_2} \kappa^{-1/2} \cdot \frac{1}{2\pi i}\oint e^{2h_2\sqrt{\kappa}\sqrt{\tau_2-\tau_1}(t+t^{-1})} t^{-2}dt.
\end{multline*}
By \eqref{SameToDifferent}, this equals $\mathbb{E}[\mathcal{G}(h_1,\eta_1,\tau_1)\mathcal{G}(h_2,\eta_2,\tau_2)]$, which completes the proof.
\end{proof}


\bibliographystyle{alpha}

\end{document}